\documentclass[onefignum,onetabnum]{siamonline220329}
\usepackage{filecontents}
\usepackage{graphicx}
\usepackage{xcolor}
\usepackage{subfig}
\usepackage{booktabs}
\usepackage{mathtools}
\usepackage{wrapfig}
\usepackage{nicefrac}
\usepackage[ruled,algo2e]{algorithm2e} %algo2e option defines algorithm environment as algorithm2e
\usepackage[export]{adjustbox}

\usepackage{enumitem}
\setlist[enumerate]{leftmargin=.5in}
\setlist[itemize]{leftmargin=.5in}

\usepackage{bm}

 \def\R{{\mathbb R}} % Reals
 \def\N{{\mathbb N}}% natural numbers

\newcommand\norm[1]{\left\lVert#1\right\rVert}
\newcommand{\pluseq}{\mathrel{+}=}
\newcommand{\T}[1]{\pmb{\mathcal{#1}}}

\newcommand{\bigOh}[1]{\mathit{O}\!\left( #1 \right)}
\newcommand{\bigOmega}[1]{\Omega\!\left(#1\right)}
\newcommand{\bigTheta}[1]{\Theta\!\left(#1\right)}
\newcommand{\lilOh}[1]{\mathit{o}\paren{#1}}
\newcommand{\size}[1]{\left|#1\right|}
\newcommand{\abs}[1]{\size{#1}}
\newcommand{\paren}[1]{\left(#1\right)}
\newcommand{\set}[1]{\left\{#1\right\}}
\DeclareMathOperator{\Vol}{Vol}
\newcommand{\floor}[1]{\left\lfloor #1 \right\rfloor}

\newcommand{\stirling}[2]{\genfrac{\{}{\}}{0pt}{}{#1}{#2}}

\newcommand{\ttsvU}[1][1]{\textsc{TTSV#1-Unord}}
\newcommand{\ttsvG}[1][1]{\textsc{TTSV#1-Gen}}
\newcommand{\ttsv}[1][]{\textsc{TTSV#1}}
\newcommand{\fftG}{\textsc{FFT-Gen}}
\newcommand{\subEx}{\textsc{Subset-Gen}}
\newcommand{\multFFT}{\textsc{multFFT}}

\usepackage{parskip}
\SetKw{Continue}{continue}

\usepackage{mathrsfs}  
\usepackage{lipsum}
\usepackage{amsfonts}
\usepackage{epstopdf}
\ifpdf
  \DeclareGraphicsExtensions{.eps,.pdf,.png,.jpg}
\else
  \DeclareGraphicsExtensions{.eps}
\fi

\usepackage{xparse}
\usepackage{amssymb,amsfonts,amsmath}
\usepackage{mathabx}
\usepackage[mathscr]{eucal}
\usepackage{amsbsy}
\usepackage{stmaryrd}

\usepackage{braket}
\usepackage{siunitx}
\usepackage{pifont}
\NewDocumentCommand{\Tn}{O{} m}{\boldsymbol{#1\mathscr{\MakeUppercase{#2}}}}
\NewDocumentCommand \X { } {\Tn{X}}
\NewDocumentCommand \A { } {\Tn{A}}
\NewDocumentCommand \LL { } {\Tn{L}} 
\NewDocumentCommand \Vc { O{} m } {{\bm{#1\mathbf{\MakeLowercase{#2}}}}}
\NewDocumentCommand \Mx { O{} m } {{\bm{#1\mathbf{\MakeUppercase{#2}}}}} 

\usepackage[normalem]{ulem}

\newcommand{\strike}[1]{}

\newsiamremark{remark}{Remark}
\newsiamremark{hypothesis}{Hypothesis}
\crefname{hypothesis}{Hypothesis}{Hypotheses}
\newsiamthm{claim}{Claim}

\headers{Tensor Methods for Nonuniform Hypergraphs}{S.G. Aksoy, I. Amburg, S.J. Young}

\title{Scalable tensor methods for nonuniform hypergraphs\thanks{
Authors listed in alphabetical order. \textbf{PNNL Information Release:} \emph{PNNL-SA-186918}
}}

\author{Sinan G. Aksoy\thanks{Pacific Northwest National Laboratory,~(\email{sinan.aksoy@pnnl.gov}, \email{ilya.amburg@pnnl.gov}, \email{stephen.young@pnnl.gov})}
\and Ilya Amburg\footnotemark[2] \thanks{Corresponding author.}
\and Stephen J. Young\footnotemark[2]%\thanks{Pacific Northwest National Laboratory, Seattle, WA, 98109, \email{sinan.aksoy@pnnl.gov}}
}

\ifpdf
\hypersetup{
  pdftitle={Tensor Methods for Nonuniform Hypergraphs},
  pdfauthor={I. Amburg, S.J. Young, and S.G. Aksoy}
}
\fi

\begin{document}

\maketitle

% REQUIRED
\begin{abstract}
While multilinear algebra appears natural for studying the multiway interactions modeled by hypergraphs, tensor methods for general hypergraphs have been stymied by theoretical and practical barriers. A recently proposed adjacency tensor is applicable to nonuniform hypergraphs, but is prohibitively costly to form and analyze in practice. We develop tensor times same vector (TTSV) algorithms for this tensor which improve complexity from $O(n^r)$ to a low-degree polynomial in $r$, where $n$ is the number of vertices and $r$ is the maximum hyperedge size. Our algorithms are implicit, avoiding formation of the order $r$ adjacency tensor.
We demonstrate the flexibility and utility of our approach in practice by developing tensor-based hypergraph centrality and clustering algorithms. We also show these tensor measures offer complementary information to analogous graph-reduction approaches on data, and are also able to detect higher-order structure that many existing matrix-based approaches provably cannot. 
\end{abstract}

% REQUIRED

% REQUIRED

\section{Introduction}
The study of hypergraphs is fraught with choices of representation. From Laplacians \cite{ bolla1993spectra, cardoso2022signless, rodri2002laplacian, zhou2006learning}, to probability transition matrices \cite{carletti2020random, chitra2019random, hayashi2020hypergraph}, to variants of incidence and adjacency matrices \cite{aksoy2020hypernetwork, cardoso2022adjacency, sole1996spectra} — there is no shortage of proposed hypergraph data structures. Despite these options, selecting among them can be challenging, as each comes with significant and sometimes nuanced limitations.  For example, adjacency, random walk, and Laplacian matrices are typically lossy in that they only contain information about the hypergraph’s clique expansion graph, thereby losing the information encoded in higher-order interactions \cite{agarwal2006higher, hayashi2020hypergraph}. In contrast, rectangular incidence matrices do faithfully model hypergraphs, but have analytical limitations: for instance, their singular values reflect information about the weighted line graph and clique expansion reductions, which do not uniquely identify the hypergraph \cite{kim2022gram}. Arguably, these challenges stem from mismatching hypergraphs, models of higher-dimensional relationships, with two-dimensional arrays.

Tensor arrays, therefore, appear a natural choice for hypergraph-native analyses. However, their application in the hypergraph setting poses immediate conceptual and computational challenges. A primary theoretical barrier is that nonuniform hypergraphs do not afford an obvious tensor representation. For this reason, despite real hypergraph-structured data nearly always exhibiting hyperedges of varying sizes, much of the existing tensor literature on hypergraphs is limited to the uniform case \cite{benson2015tensor, benson2017spacey, sharmajoint}, relies on augmenting the hypergraph with auxiliary nodes \cite{ouvrard2017adjacency,zhen2021community}, or synthesises a collection of differently sized tensors for each hyperedge size \cite{ke2019community}. One notable exception of using tensors to {\it directly} study nonuniform hypergraphs, however, is the adjacency tensor recently proposed by Banerjee, Char, and Mondal \cite{BANERJEE201714}. Loosely speaking, this tensor encodes nonuniform hyperedeges by ``inflating'' each to the maximum hyperedge size $r$. This yields an order $r$, $n$-dimensional tensor, where $n$ is the number of vertices. Consequently, the nonuniform adjacency tensor solves a conceptual challenge but poses a computational one: its explicit formation and analysis is intractable for nearly any hypergraph data with non-trivially sized hyperedges, since fundamental tensor operations like tensor times same vector (\ttsv) have cost $O(n^r)$. 
%since the maximum hyperedge size (of data exhibiting heavily-tailed degree distributions) often grows linearly in, or as a fractional power of, $n$ \cite{leskovec2007graph}. 

In this work, we focus on ameliorating these computational challenges to enable use of the hypergraph adjacency tensor in practice.  Our main focus is creating efficient algorithms for \ttsv. In particular, we drastically speed up the tensor times same vector in all modes but one (\ttsv[1]) operation from $O(n^r)$ to a low-degree polynomial in $r$. We perform an analogous speedup for tensor times same vector in all modes but two (\ttsv[2]) using an approach that easily generalizes to tensor times same vector in all modes but $k$ (\ttsv[k]). Moreover, our methods are implicit and tensor-free, avoiding formation of the costly order-$r$ tensor. We achieve these improvements using combinatorial methods that exploit the nuanced symmetry of the hypergraph adjacency tensor. We derive best- and worst-case complexity bounds of our algorithms, and supplement these analytical results with timing experiments on real data. 

Since \ttsv\ is a workhorse in many tensor algorithms such as CP decomposition and tensor eigenvector computation \cite{benson2019three,benson2019computing,benson2015tensor,benson2017spacey,kolda2015numerical,kolda2009tensor,sherman2020estimating}, our algorithms enable a host of tensor-based hypergraph analytics.  We illustrate this by proposing simple tensor-based centrality and clustering algorithms where \ttsv\ is the primary subroutine. 
For centrality, we apply recent nonlinear $Z$- and $H$-eigenvector formulations \cite{benson2019three} to nonuniform hypergraphs, whose existence is guaranteed by the Perron-Frobenius theorem for tensors \cite{qi2017tensor}. For clustering, we outline an approach that uses fast CP decomposition to obtain an embedding for the hypergraph, which is then fed into $k$-means~\cite{hartigan1979algorithm}, or any metric space approach, for clustering.  
We then study these measures experimentally, showing each offers complementary node importance information on data. 
Moreover, we show these tensor measures detect differences in structured hypergraphs with identical underlying graph information, as given by their weighted clique and line graphs.  
In contrast to many existing hypergraph methods, this means tensor approaches enabled by our algorithms analyze multiway interactions in hypergraph data directly -- without reducing them to groups of pairwise interactions modelled by graphs.

The paper is structured as follows: Section \ref{sec:prelims} reviews the necessary preliminaries. Section \ref{sec:ttsvs} presents our algorithms for \ttsv[1].  As the details of the algorithms and analysis for \ttsv[2] are similar to that of \ttsv[1], we defer their discussion to Section \ref{sec:TTSV2}.  Section \ref{sec:applications} applies these algorithms to develop a tensor-based approach for nonuniform hypergraph centrality and clustering. Section \ref{sec:conclusion} concludes and highlights avenues for future work.

\section{Preliminaries}\label{sec:prelims}
A {\it hypergraph} $H=(V,E)$ is a set $V$ of $n$ vertices and a set $E$ of $m$ hyperedges, each of which is a subset of $V$. The {\it degree} of a vertex $v$ and hyperedge $e$ is $d(v)=|\{e \in E: v \in e\}|$ and $d(e)=|e|$, respectively. The {\it rank} of a hypergraph is $\max_{e} d(e)$ and if $d(e)=k$ for all $e\in E$, we call the hypergraph {\it $k$-uniform}. The {\it volume} of a hypergraph is $\mbox{Vol}(H)=\sum_{v \in V}d(v)=\sum_{e \in E}d(e)$. The set of hyperedges to which $v$ belongs is denoted by $E(v)$ while $E(u,v)$ denotes those to which both $u$ and $v$ jointly belong. 
The {\it clique expansion} of a hypergraph $(V,E)$ is the graph on $V$ with edge set  $\{\{u,v\} \in V \times V: u,v\in e \mbox{ for some } e \in E\}$. 
For more basic hypergraph terminology, we refer the reader to \cite{aksoy2020hypernetwork, berge1984hypergraphs}. 

A tensor of order $r$ is an $r$-dimensional array. Lowercase bold letters denote vectors, e.g. $\Vc{a}$, while uppercase bold letters in regular and Euler script denote matrices and tensors, e.g. $\Mx{X}$ and $\T{X}$, respectively. For a tensor $\T{X}$, the value at index $i_1,\dots,i_r$ is given by $\T{X}_{i_1,\dots,i_r}$. In general, we assume the vertices of a hypergraph are indexed by $[n]=\{1,\dots,n\}$ and that each index in a tensor $\T{X}$ has $n$ components, $i_1,\dots, i_r\in [n]^r$. The notation $\Vc{a}^{[k]}$, $\Vc{a} \odot \Vc{b}$, and $\Vc{a} \oslash \Vc{b}$ denote the elementwise power, product, and division operations, respectively. Following standard notation from generating functions (see for instance \cite{wilf}), we will denote the coefficient of $t^r$ in the generating function $f(t)$ by $f[t^r]$. Lastly, we use Bachmann-Landau asymptotic notation. Our algorithm complexities are multivariate, and depend on the rank $r$, number of hyperedges $m$, and volume $\mbox{Vol}(H)$ of the hypergraph $H$.

Tensors have long been used to encode adjacency in uniform hypergraphs. In particular, a rank $r$ uniform hypergraph, its adjacency tensor is order $r$ with
\[
\A_{v_1,\dots,v_r}=\begin{cases} w  &\mbox{ if } \{v_1,\dots,v_r\}\in E\\ 0 & \mbox{ otherwise }
\end{cases} ,
\]
where $w$ refers to a chosen weight, such as $w=1$ or $1/(r-1)!$. 
Recently, \cite{BANERJEE201714} generalized this to nonuniform hypergraphs. Underlying this definition is a concept we call hyperedge {\it blowups}.

\begin{definition}\label{def:blowup}
Given a hyperedge $e$ of a rank $r$ hypergraph, we call the sets
\begin{align*}
\beta(e) &= \{i_1 \dots i_r \in e^r : \mbox{ for each } v\in e, \mbox{ there is } j \mbox{ such that } i_j=v\}, \\
\kappa(e) &= \{x: x \mbox{ is a size } r \mbox{ multiset with support } e\},
\end{align*}
the ``blowups" and ``unordered blowups" of hyperedge $e$, respectively.
\end{definition}

For example, for $e=\{1,3\}$ in a rank $3$ hypergraph, we have
\begin{align*}
\beta(e) &= \{(1,1,3),(1,3,1),(1,3,3),(3,1,1),(3,1,3),(3,3,1)\}, \\
\kappa(e) &= \{\{1,1,3\},\{1,3,3\}\}.
\end{align*}
The nonuniform hypergraph adjacency tensor \cite{BANERJEE201714} places nonzeros in positions corresponding to blowups of hyperedges, with values weighted by the size of that hyperedge's blowup set:
\begin{definition}[Nonuniform Hypergraph Adjacency Tensor \cite{BANERJEE201714}]
For a rank $r$ hypergraph, its adjacency tensor $\A$ is order $r$ and defined elementwise for each hyperedge $e$,
\[
\A_{p_1\dots p_{r}}=\begin{cases}w_e &\mbox{ if } p_1\dots p_r \in \beta(e) \\
0 &\mbox{ otherwise}
\end{cases},
\]
where $w_e$ denotes a chosen hyperedge weighting function.
% where $\beta(e)=\{p_1 \dots p_r \in e^r : \mbox{ for each } v \in e \mbox{ there exists } j  \mbox{ where } p_j=v \}$, and  $\alpha(e,r)=|e|!{r\brace |e|}$ where ${r\brace |e|}$ is the Stirling number of the second kind.
\end{definition}
We note that while \cite{BANERJEE201714} takes $w_e=\frac{|e|}{|\beta(e)|}$ so that \ttsv[1] with the all ones vector yields the degree sequence, our algorithms will not require this choice. %\color{blue} Our main contribution is drastically speeding up two fundamental operations for $\A$: tensor times same vector in all modes but one (\ttsv[1]) and tensor times same vector in all modes but two (\ttsv[2]).\color{black} 
Furthermore, in addition to adjacency tensors, our methods also easily adapt to the Laplacian tensors in \cite{BANERJEE201714}, as we illustrate in Section \ref{subsec:clustering}.
Lastly, a different approach to building nonuniform hypergraph adjacency tensors relies on inserting copies of an auxiliary ``dummy" vertex within each hyperedge \cite{zhen2021community}, rather than using those already present in the hyperedge via blowups. While we focus on the blowup approach, our algorithms easily adapt to this case as well. 
The \ttsv\ operations we tailor for $\A$ form the backbone of many tensor algorithms such as tensor decomposition~\cite{kolda2015numerical,kolda2009tensor,sherman2020estimating} and eigenvector computation~\cite{benson2019three, benson2019computing,benson2015tensor,benson2017spacey}, and are defined in general as follows: given an order $r$ tensor $\X$, and a vector $\Vc{b}\in\R^n,$ 
$\Tn{X}\Vc{b}^{r-1}\in\R^n$ denotes the TTSV1 operation given by 
\begin{align}\label{eq:ttsv1}
\left[\Tn{X}\Vc{b}^{r-1}\right]_{i_1} &= \sum_{i_2=1}^n\cdots\sum_{i_r=1}^n\X_{i_1,\dots,i_r}\prod_{k=2}^r\Vc{b}_{i_k},
\end{align}
and TTSV2, denoted $\X \Vc{b}^{r-2}\in\R^{n\times n}$, is given by
\begin{align}\label{eq:ttsv2}
\left[\X \Vc{b}^{r-2}\right]_{i_1,i_2} &= \sum_{i_3=1}^n\cdots\sum_{i_r=1}^n\X_{i_1,\dots,i_r}\prod_{k=3}^r\Vc{b}_{i_k}.
\end{align}
Given the TTSV2 matrix, the TTSV1 vector is easily obtained through right-multiplication by $\Vc{b}$.
%Note it is easy to obtain the TTSV1 vector given the TTSV2 matrix through right-multiplication by $\Vc{b}$. 
However, whenever the full TTSV2 is unnecessary, it suffices to compute TTSV1 directly. 
%However, the full TTSV2 is not needed in all applications, and it often suffices to compute TTSV1 directly.
Lastly, we note computing TTSV without explicitly forming the tensor is an example of an {\it implicit} tensor operation, and has been recently studied in the context of moment tensors of Gaussian mixture models \cite{pereira2022tensor, sherman2020estimating}.
\section{Tensor times same vector for hypergraphs} \label{sec:ttsvs}
%Our core contribution lies in operationalizing the Banerjee adjacency tensor framework for general hypergraphs to enable efficient tensor operations with the tensor encoding of general hypergraphs. 
%\sga{New title: Implicit tensor-time-same vector for hypergraphs}
We now develop efficient, implicit \ttsv\ methods for analyzing the hypergraph adjacency tensor $\A$.  Rather than explicitly constructing, storing, or accessing elements of $\A$, our algorithms directly facilitate tensor operations on the input hypergraph. 
We divide our work into two approaches: first, we present simple algorithms which achieve speedup over the na\"ive approach by leveraging several combinatorial observations on the unordered blowups discussed in Definition \ref{def:blowup}. Then, we further improve upon these algorithms by using generating function theory. Code for our TTSV algorithms is available at \url{https://github.com/pnnl/GENTTSV}, in Python for hypergraph libraries HyperNetX \cite{praggastis2023hypernetx}, HypergraphX \cite{lotito2023hypergraphx}, and XGI \cite{landry2023xgi}, and Julia for SimpleHypergraphs \cite{antelmi2019simplehypergraphs}.

\subsection{Unordered blowup approach}

This approach for computing \ttsv[1] for $\A$ will use the following basic combinatorial facts.

\begin{lemma}\label{lem:comb}
Let $e=\{v_1,\dots,v_k\}$ be a hyperedge of a rank $r$ and $\beta(e)$ and $\kappa(e)$ as defined in Definition $\ref{def:blowup}$. Then
\begin{itemize}
    %\item[(a)] $|\beta(e)|=|e|!\cdot {r\brace |e|}$, where where ${r\brace |e|}$ denotes the Stirling number of the second kind \footnote{An explicit formula for Stirling numbers of the second kind is ${n\brace k}=\frac{1}{k!}\sum_{i=0}^k (-1)^i {k \choose i} (k-i)^n$.}.
    \item[(a)] $|\beta(e)|=|e|!\cdot \stirling{ r}{ |e|}$, where where $\stirling{r}{ |e|}$ denotes the Stirling number of the second kind \footnote{An explicit formula for Stirling numbers of the second kind is $\stirling{n}{k}=\frac{1}{k!}\sum_{i=0}^k (-1)^i \binom{k}{i} (k-i)^n$.}.
    \item[(b)] $|\kappa(e)|= \binom{r-1}{r-|e|}$.
    \item[(c)] For given vertices $u \in e$ and $x\in\kappa(e)$, the number of blowups $i_1,\dots,i_r \in \beta(e)$ that have support equal to that of $x$ with $i_j=u$ is given by the multinomial coefficient
\[
        \phi_1(x,u) \coloneqq
        \binom{r-1}{m_x(v_1),\dots,m_x(u)-1,\dots, m_x(v_k)} 
\]
   where $m_x(w)$ denotes the multiplicity of node $w$ in $x$.
\end{itemize}
%\cite{galuppi2022spectral}
\end{lemma}

%\begin{figure}
%  \vspace{0pt}  
\phantomsection\label{alg:TTSV1U}
  \begin{algorithm}[H] \small
  \LinesNotNumbered 
 %\SetLine
\KwData{rank $r$ hypergraph $(V,E,w)$, vector $\Vc{b}$}
\KwResult{$\A\Vc{b}^{r-1}=\Vc{s}$}
\For{$v\in V$}{
$c \gets 0$\;

\For{$e\in E(v)$}{
%$\alpha=\alpha(|H(e)|,r)$\;
\For{$x\in\kappa(e)$}{
$c\pluseq \displaystyle w_e \frac{\phi_1(x,v)}{\Vc{b}_v}  \prod\limits_{u \in x} \Vc{b}_u$\;
}
}
$\Vc{s}_v \gets c$\;
}

\Return{$\Vc{s}$}
\caption*{\textsc{TTSV1-Unord}: \ttsv[1] via unordered blowups}
\end{algorithm}

%\caption{Function for calculating TTSV1 implictly from the hypergraph adjacency tensor $\A$ via the unordered blowup approach.}\label{fig:algs}
%\end{figure}

\hyperref[alg:TTSV1U]{\ttsvU} presents the algorithm for the unordered blowup approach. 
The main idea is to exploit the nonzero pattern in $\A$. 
Recalling that $i_1$ corresponds to a node index when applying the TTSV1 vector given in Eq. \ref{eq:ttsv1} to $\A$, the only nonzeros that contribute to the sum for this component correspond to blowups of all hyperedges to which node $i_1$ belongs. This observation significantly reduces the cost of \ttsv[1] from the naive $O(n^r)$ approach in the absence of the special structure induced by the hypergraph, and leads to a less na\"ive algorithm: for every vertex, iterate through its hyperedges and update the sum over all elements of the corresponding blowups. However, this approach still requires explicitly enumerating blowups, which is extremely costly. Instead, \hyperref[alg:TTSV1U]{\ttsvU} only considers {\it unordered} blowups. By Lemma \ref{lem:comb}(c), there are exactly $\phi_1(x,v)$ many blowups with $v$ in a fixed position that correspond to a given unordered blowup $x \in \kappa(e)$. The correctness of \hyperref[alg:TTSV1U]{\ttsvU} immediately follows. We note the cost savings here occurs {\it per hyperedge}, yielding speedups of several orders of magnitude for real datasets with larger $r$.

\begin{proposition}\label{P:TTSV}
Let $H=(V,E)$ be a rank $r$ hypergraph with $m$ edges then  \hyperref[alg:TTSV1U]{\ttsvU} runs in time $\bigTheta{\sum_{e \in E} r\size{e} \binom{r-1}{{r-\size{e}}}}$. Let $\epsilon = \min\set{\frac{\Vol(H)}{m}, 1- \frac{\Vol(H)}{m}}$, then this running time is at least 
\[ \bigOmega{r^2 m} \quad \textrm{and at most} \quad \bigOh{\epsilon m r^{\nicefrac{3}{2}}2^r}.\]
\end{proposition}
\begin{proof}
Calculating the inner most terms of the sum runs in $\bigTheta{r}$ and thus the run time is 
\[ \sum_{v \in V} \sum_{e \in E(v)} \sum_{x \in \kappa(e)} \bigTheta{r} = \sum_{e \in E} \sum_{v \in e} \sum_{x \in \kappa(e)} \bigTheta{r} = \sum_{e \in E} \bigTheta{ r \size{e} \binom{ r- 1}{r - \size{e}}}.\]
Thus, we have that the run time is completely determined by the sizes of the edges of $H$.  Now noting that, by algebraic manipulation, we have that $r \size{e} \binom{r - 1}{r - \size{e}} = \size{e}^2 \binom{r}{\size{e}}$, we define $f: \set{2,\ldots, r} \rightarrow \N$ by $f(k) = k^2 \binom{r}{k}$.  We observe that 
\[ \frac{f(k+1)}{f(k)} = \frac{(k+1)^2 \binom{r}{k+1}}{k^2\binom{r}{k}} = \frac{(k+1)^2 k! (r-k)!}{k^2 (k+1)! (r-k-1)!} = \frac{k+1}{k}\frac{r-k}{k}.\] Thus the maximum of $f$ occurs at $k^* = \floor{\frac{r}{2}} + 1$ and $f$ is monotonically increasing below $k^*$ and monotonically decreasing above $k^*$.
As $f(2), f(r) \in \bigTheta{r^2}$, we have that the run time of \hyperref[alg:TTSV1U]{\ttsvU} is bounded below by $\bigTheta{r^2m}$.  

For the upper bound, let $e_1, \ldots, e_m$ be a sequence of integers such that $\sum_i e_i = \Vol(H)$ and $\sum_i f(e_i)$ is maximized.  By the monotonicity of $f$, either $e_i \leq k^*$ for all $i$ or $e_i \geq k^*$ for all i.  Let $C = \floor{2\sqrt{r}}$ and suppose $k$ and $\ell$ are such that $\abs{k - k^*}, \abs{\ell - k^*} \geq C$, then we have  
\begin{align*}
    \frac{f(k) + f(\ell)}{f(k^*)} = \frac{k^2 \binom{r}{k} + \ell^2 \binom{r}{\ell}}{(k^*)^2\binom{r}{k^*}}
    \leq \frac{2r^2 \binom{r}{k^* - C}}{(k^*)^2 \binom{r}{k^*}} 
    &\leq 8 \frac{ (k^*)! (r- k^*)!}{ (k^* - C)! (r - k^* + C)!} \\
    &= 8 \prod_{j = 1}^{C} \frac{k^* - C + j}{r - k^* + j} \\
    &\leq 8\paren{ \frac{k^*}{r - k^* + C}}^{C} \\
    &= 8\paren{ 1 - \frac{r - 2k^* + C}{r -k^* + C}}^{C} \\
    & \leq 8\paren{1 - \frac{C-1}{\frac{r}{2} + C}}^{C} \\
    &\leq 8 \exp\left( - \frac{C(2C-2)}{r + 2C}\right) \leq 1.
\end{align*}
This implies that, by the maximality of $e_1,\ldots,e_m$, that $e_i \in \set{2,r} \cup [k^* - C, k^* + C]$ for all $i$.  Further, since $f(k) = \bigTheta{f(k^*)} = \bigTheta{r^{\nicefrac{3}{2}}2^r}$ for all $k^* - C \leq k \leq k^* + C$, this gives the desired upper bound on the runtime of 
\hyperref[alg:TTSV1U]{\ttsvU}.
\end{proof}

\subsection{Generating function approach} 
% \stephen{ start with the inner loop as the bulk of the computational time, think through what happens with a small edge, then this leads naturally to the generating function approach }\sga{This exposition is nice}\newline
The run time of \hyperref[alg:TTSV1U]{\ttsvU} is dominated by the inner-most loop, iterating over unordered blowups $\kappa(e)$ for all edges $e$.  Indeed, for most real datasets, where the size of a typical edge is much smaller than the largest edge, this is a significant bottleneck as the size of the this set scales exponentially with the size of the edge, i.e.\ $|\kappa(e)| = O(r^{|e|})$.  However, by taking a generating function approach the inner loop can be considerably simplified and accelerated.  To illustrate this, consider the 
summation over the inner loop for the case when $e = \set{u,v} \in E(v)$,
\begin{align*}
    \sum_{x \in \kappa(\set{u,v})} w_e \frac{\phi_1(x,v)}{\Vc{b}_v} \prod_{i \in x} \Vc{b}_i &= \frac{w_e}{\Vc{b}_v} \sum_{k_u = 1}^{r-1} \phi_1( u^{k_u} v^{r-1 - k_u}, v) \Vc{b}_u^{k_u} \Vc{b}_v^{r-k_u} \\
    &= w_e \sum_{k_u = 1}^{r-1} \frac{(r-1)!}{k_u! (r-1-k_u)!} \Vc{b}_u^{k_u}\Vc{b}_v^{r-1-k_u} \\
    &= w_e (r-1)! \sum_{k_u = 1}^{r-1} \frac{1}{k_u!} \Vc{b}_u^{k_u} \frac{1}{(r-1-k_u)!} \Vc{b}_v^{r-1-k_u}. \tag{\textrm{\textasteriskcentered}}\label{eq:sum}
\end{align*}
We note that the summation in (\ref{eq:sum}) can be viewed as a convolution of two sequences 
\[ 0,\frac{1}{1!}\Vc{b}_u^1, \frac{1}{2!}\Vc{b}_u^2, \ldots, \frac{1}{(r-1)!} \Vc{b}_u^{r-1} \quad \textrm{and} \quad \frac{1}{0!} \Vc{b}_v^0, \frac{1}{1!}\Vc{b}_v^1,  \frac{1}{2!}\Vc{b}_v^2, \ldots, \frac{1}{(r-1)!}\Vc{b}_v^{r-1}.\]
Alternatively, if we define $u_r(t) = \sum_{k=1}^{r-1} \frac{1}{k!}\Vc{b}_u^kt^k$ and $v_r(t) = \sum_{k=0}^{r-1} \frac{1}{k!}\Vc{b}_v^kt^k$, the summation (\ref{eq:sum}) may be thought of as the coefficient of $t^{r-1}$ in $u_r(t)v_r(t)$.  Recall that the coefficient of $t^r$ in the generating function $f(t)$ is denoted by  $f[t^r]$.\strike{ function $f(t)$ we will denote the coefficient which we denote $(u_r(t)v_r(t))[t^{r-1}]$. }
In fact, for any $d\geq 0$ the summation is equal to \strike{$\paren{u_{r+d}(t) v_{r+d}(t)}[t^{r-1}]$}$\paren{u_{r+d}v_{r+d}}[t^{r-1}]$ and thus taking 
\[ u(t) = \sum_{k=1}^{\infty} \frac{1}{k!}\Vc{b}_u^k t^k = \exp({\Vc{b}_ut}) - 1 \quad \textrm{and} \quad v(t) = \sum_{k=0}^{\infty} \frac{1}{k!} \Vc{b}_v^k t^k = \exp({\Vc{b}_vt}),\]
we have that \strike{$(u(t)v(t))[t^{r-1}]$}$(u v)[t^{r-1}]$ is the value of the inner sum for an arbitrary choice of largest edge size $r$.\strike{Extending the above argument to an arbitrary edge $e \in E(v)$,} This argument may be easily extended to an arbitrary edge $\set{v,u_1,\ldots, u_{\ell}} = e \in E(v)$, by noting that the contribution to inner sum  by blowups of $e$ with $k_v$ copies of $v$ and $k_{1}, \ldots, k_{\ell}$ copies of $u_1,\ldots, u_{\ell}$ is given by 
\[ w_e\frac{(r-1)!}{(k_v)! \prod_{j=1}^{\ell} (k_j)!} \Vc{b}_{v}^{k_v} \prod_{j=1}^{\ell} \Vc{b}_{u_j}^{k_j} = w_E(r-1)! \frac{\Vc{b}_v^{k_v} }{(k_v)!} \prod_{j=1}^{\ell} \frac{ \Vc{b}_{u_j}^{k_j}}{(k_j)!}. \]
Now by summing over valid choices of $k_v$ and $k_1,\ldots,k_{\ell}$, we get that the total contribution of the inner loop from the edge $e$ is given by
\[ w_e (r-1)! \sum_{k_v=0}^{r-\ell-1} \sum_{\substack{1 \leq k_j \leq r-\ell \\ k_v + \sum_j k_j = r}}  \frac{\Vc{b}_v^{k_v} }{(k_v)!} \prod_{j=1}^{\ell} \frac{ \Vc{b}_{u_j}^{k_j}}{(k_j)!}.\]
As above, if we introduce the variable $t$ this can be viewed as the coefficient of $t^{r-1}$ in a generating function, specifically 
\[
\paren{w_e (r-1)! \sum_{k_v = 0}^{\infty} \sum_{k_1,\ldots, k_{\ell} = 1}^{\infty} \frac{\Vc{b}_v^{k_v}t^{k_r} }{(k_v)!} \prod_{j=1}^{\ell} \frac{ \Vc{b}_{u_j}^{k_j}t^{k_j}}{(k_j)!} } [t^{r-1}]. \]

Thus, by interchanging the products and summations we have that the inner sum in \hyperref[alg:TTSV1U]{\ttsvU} is given by \[w_e (r-1)!\paren{\exp({\Vc{b}_vt}) \prod_{u \in e - \set{v}} \paren{\exp({\Vc{b}_ut}) - 1 }}[t^{r-1}].\]

Alternatively, the expression for the inner sum can be derived analytically from the generating function representation for \ttsv. Specifically, using the same approach as above we have that 
\[ \Tn{X}\Vc{b}^r = \paren{\sum_{e \in E} w_e r! \prod_{v \in e} (\exp(\Vc{b}_vt)-1)}[t^r]. \] 
Since, as shown in \cite[Lemma 3.1 and 3.3]{kolda2011shifted} for $k=1$ and $2$, any symmetric tensor $\Tn{X}$ and vector $\Vc{b}$ satisfy $\nabla^k \left(\Tn{X}\Vc{b}^r\right) = \frac{r!}{(r-k)!} \Tn{X}\Vc{b}^{r-k}$, we have that 
\[ \Tn{X}\Vc{b}^{r-k} = \paren{\sum_{e \in E} w_e (r-k)! \nabla^k \prod_{v \in e} (\exp(\Vc{b}_vt)-1)}[t^{r-k}].\]
For example, the $(u,u)$ diagonal term of \ttsv[2] are given by
\[\paren{\sum_{e \in E} w_e (r-2)! \frac{\partial^2}{\partial \Vc{b}_u^2} \prod_{v \in e} \left(\exp(\Vc{b}_vt)-1\right)}[t^{r-2}]\]\[ = \paren{\sum_{e \in E(u)} w_e (r-2)! \exp(\Vc{b}_ut)  \prod_{v \in e - \set{u}} \left(\exp(\Vc{b}_vt)-1\right)}[t^{r-2}]. \]

%\begin{figure}
%\centering

\begin{wrapfigure}[16]{l}{0.35\textwidth}
    \centering
    \includegraphics[width = 1.\linewidth]{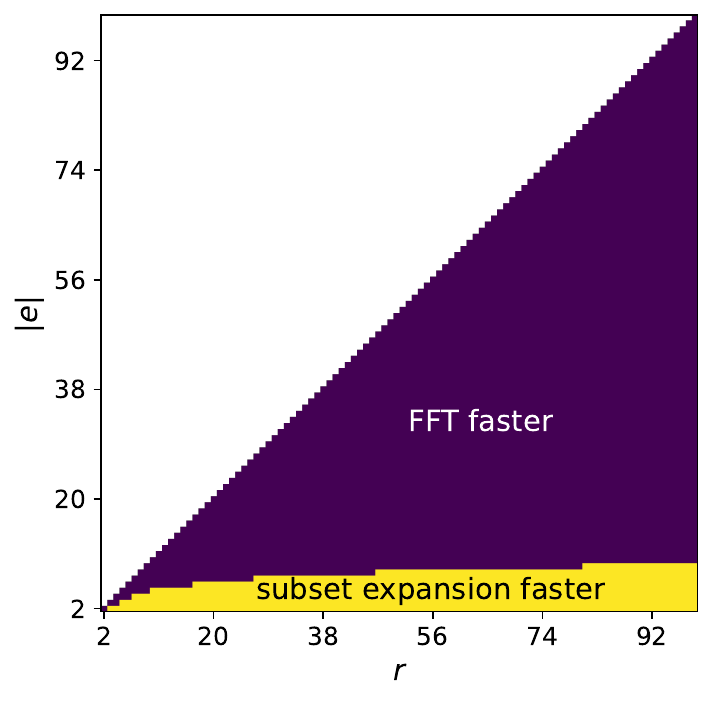}
    \caption{Efficiency of FFT versus subset expansion for generating function evaluation. }%is more efficient with subset expansion in the yellow region, and with FFT in the purple region.\sga{not so much white space}}
    \label{fig:faster}
\end{wrapfigure}

While this generating function simplifies the exposition of these algorithms, in practice we require an efficient means of extracting the appropriate coefficient from the generating function. 
% \stephen{this transitions sucks}. \sga{I think it is fine. I added ``Accordingly" :-)}  
Accordingly, we present two different options for evaluating the generating function approach, a subset expansion approach, \hyperref[alg:SubEx]{\subEx},  which runs in time $\bigTheta{\paren{\size{e} + \log_2(r)}2^{\size{e}}}$,
%\ilya{My implementation uses dynamic programming to compute the subset sums first, so it is just $\bigTheta{\paren{\log_2(r)}2^{\size{e}}}$ in time, at the cost of a little more storage}
and a modification of the fast Fourier transform (FFT)~\cite{cooley1965algorithm}  approach to polynomial multiplication, \hyperref[alg:fftG]{\fftG}, which runs in time $\bigTheta{\size{e} r\log_2(r)}$. 
We note the runtimes of these two algorithms are asymptotically equivalent for edges of size $\log_2(r) + \log_2\log_2(r) + \bigTheta{1}$, with \hyperref[alg:SubEx]{\subEx} asymptotically faster for smaller edges and \hyperref[alg:fftG]{\fftG} for larger edges.
%that the subset expansion approach is asymptotically faster for $\size{e} < \log_2(r) + \log_2 \log_2(r)$ and for larger edges the FFT algorithm is more efficient.  Thus, if we assume the subset algorithm is only used for sufficiently small edges, its running time is $\bigOh{\log_2(r)2^{\size{e}}}$.  
For convenience, we assume all edges of size at most  $\log_2(r) + \log_2\log_2(r)$ are evaluated using \hyperref[alg:SubEx]{\subEx} and larger edges are evaluated via \hyperref[alg:fftG]{\fftG}. 
Figure~\ref{fig:faster} plots the edge size region where subset expansion outperforms FFT (and vice-versa). 
It is interesting to note that difference in run-times between the two approaches can be understood by which computationally efficient sub-routines are being used. For instance, the \hyperref[alg:fftG]{\fftG} relies on the FFT as a computationally efficient subroutine, and as a consequence must take the time to form all the degree $r$ polynomials associated with the edge $e$ and operate on these polynomials (which are larger than the original edge).  In contrast to this approach, \hyperref[alg:SubEx]{\subEx} relies on the fact that arbitrary coefficients of exponential generating functions can be efficiently evaluated.  In particular, \hyperref[alg:SubEx]{\subEx} relies on the fact that 
 $e^{at}\prod_{i=1}^{k}\paren{e^{b_it} - 1}$ can be expanded to 
 \[ e^{at} \sum_{S \subseteq [k]} (-1)^{\size{\overline{S}}}\prod_{s \in S}e^{b_st} = \sum_{S \subseteq [k]} (-1)^{k-\size{S}} e^{(a + \sum_{s \in S} b_s)t}.\]  
%  \[ \sum_{S \subseteq [k]} e^{at}(-1)^{\size{\overline{S}}}\prod_{s \in S}e^{b_st} = (-1)^{k-\size{S}} e^{(a + \sum_{s \in S} b_s)t}.\]  
 The $t^r$ coefficient of each of these summands can then be trivially evaluated.  Thus, there is an exponential dependence on the edge size and minimal dependence on the maximum edge size $r$.

 While not explored in this work, we note that further asymptotic speed-ups are possible by hybridizing these two approaches.  Specifically, by partitioning the edge $e$ into small sets $e_1, \ldots, e_k$ of size approximately $\log\log(r)$, the full generating function for these can be calculated for the sets individually using a variant of the \hyperref[alg:SubEx]{\subEx} approach and then this set of generating functions could be combined using the FFT.
%\stephen{If we are going to include this figure we should make sure that it agrees with right trade off point}
% \stephen{This section could probably be expanded a bit} \sga{With what? Seems complete to me but could be missing something.}

\begin{center}
\begin{minipage}[t]{0.48\linewidth}
  \vspace{0pt}  
  \phantomsection\label{alg:fftG}
  \begin{algorithm}[H] \small
\LinesNotNumbered 
\KwData{$a, b_1, \ldots, b_k \in \R$ and $r \in \N$.}
\KwResult{$\paren{\exp({at}) \displaystyle\prod_i \paren{\exp({b_it}) - 1}}[t^r]$}
{$f \gets \displaystyle\sum_{j=0}^r \frac{a^j}{j!} t^j$}\;\\
\For{ $i = 1, \ldots, k$}{
{$g \gets \displaystyle\sum_{j=1}^r \frac{b_i^j}{j!}t^j$} \; 

{$f \gets \textsc{multFFT}(f,g)$}\; 

{$f \gets \displaystyle\sum_{j=0}^r f[t^j] t^j$}\; 
}
\Return{$f[t^r]$}
\caption*{\fftG: FFT evaluation of generating function}
\end{algorithm}
\end{minipage} \hfil
\begin{minipage}[t]{0.48\linewidth}
 \vspace{0pt}
 \phantomsection\label{alg:multFFT}
\begin{algorithm}[H] \small
  \LinesNotNumbered 
 %\SetLine
\KwData{polynomials $f(t)$ and $g(t)$}
\KwResult{$f(t) g(t)$}
{$c(t) \gets FFT^{-1}(FFT(g(t))\odot FFT(f(t)))$\;} \\
\Return{$c(t)$}
\caption*{\multFFT: Multiplication using the FFT}
\end{algorithm}
%\vspace{-7pt}
\vspace{4pt}
\phantomsection\label{alg:SubEx}
\begin{algorithm}[H]
    \small
    %\LinesNotNumbered
    \KwData{$a,b_1,\ldots,b_k \in \R$ and $r \in \N$}
    \KwResult{$\paren{\exp({at}) \prod_i \paren{\exp({b_it}) - 1}}[t^r]$}
    {$c \gets 0$\;} \\% 
\For{$S \subseteq [k]$}{
        {$c \pluseq (-1)^{k - \size{S}} \paren{a + \sum_{s \in S} b_s }^r$}\;
    }
    \Return{\nicefrac{c}{r!}}
    \caption*{\subEx: Subset enumeration evaluation of generating function}  
    \label{alg:subset_gen}
    \end{algorithm}
\end{minipage}
%\caption{Subalgorithms for the FFT and subset enumeration approaches to evaluate generating functions}\label{fig:algs_sub}
\end{center}

% \begin{figure}
%     \centering
%     \begin{algorithm}[H]
%     \small
%     \LinesNotNumbered
%     \KwData{$a,b_1,\ldots,b_k \in \R$ and $r \in \N$}
%     \KwResult{$\paren{e^{at} \prod_i \paren{e^{b_it} - 1}}[t^r]$}
%     $c \gets 0$\; \\
%     \For{ S \in 2^{[k]}}{
%         $c \pluseq (-1)^{k - \size{S}} \paren{a + \sum_{s \in S} b_s }^r$\;
%     }
%     \Return{ \nicefrac{c}{r!}}
%     \caption{Subset enumeration evaluation of generating function}  
%     \label{alg:subset_gen}
%     \end{algorithm}

% \end{figure}

\begin{proposition}\label{P:TTSVgen}
Let $H=(V,E)$ be a rank $r$ hypergraph with $m$ edges and let $k^* = \log_2(r) + \log_2\log_2(r)$. The asymptotic run-time of $\hyperref[alg:TTSV1G]{\ttsvG}$ is given by 
\[ \sum_{\substack{e \in E \\ \size{e} \leq k^*}} \size{e} 2^{\size{e}}\log_2(r) + \sum_{\substack{e \in E \\ \size{e} > k^*}} \size{e}^2 r\log_2(r).\] If, in addition,  the average hyper edge size $\delta = \frac{\Vol{H}}{m}$ is such that there is some constant $\epsilon > 0$ such that $2+\epsilon < \delta < (1-\epsilon)r$, \hyperref[alg:TTSV1G]{\ttsvG} runs in time at least 
\[ \begin{cases} \bigOmega{\Vol(H)2^{\delta}\log_2(r)} & \delta < \frac{1}{2}k^* \\ \bigOmega{\Vol(H)2^{(1-\lilOh{1})\delta}\log_2(r)} & \frac{1}{2} k^* \leq \delta \leq k^*  \\
\bigOmega{\Vol(H)\delta r\log_2(r)} & k^* < \delta
\end{cases} \]
and at most  
\[ \bigOh{r^2\log_2(r) \Vol(H)}.\]
\end{proposition}

\begin{proof}
As in the proof of Proposition \ref{P:TTSV}, the run time depends on the hyperedge size sequence, yielding  
 $\bigTheta{\sum_{e \in E} \size{e} h(\size{e},r)}$, where $h(k,r)$ is the run time of extracting the coefficient of $t^{r-1}$ in the generating function associated with an edge of size $k$. If for every edge, we use the faster of \hyperref[alg:fftG]{\fftG} and \hyperref[alg:SubEx]{\subEx}, this gives that 
\[ h(k,r) = \begin{cases} \log_2(r)2^k & k \leq k^* \\ kr\log_2(r) & k^* < k\end{cases}. \]
As the shape of the edge size distribution has different effects for ``small" edges (those of size at most $k^*$) and ``large" edges (those with size at least $k^*$), we consider each of these distributions separately.  To that end, let $m_S$ be the number of small edges and let $m_F$ be the number of large edges, with $\delta_S$ and $\delta_F$ denoting the respective average edge sizes.  We note that $m = m_S + m_F$ and $\delta_S m_S + \delta_F m_F = \Vol(H) = \delta m$.   We first note that by standard results, the run time for large edges is minimized when all edges have the same size, $\delta_F$, and maximized when all edges have size either $k^*$ or $r$.

We now consider the extremal run times for the small edges and show that local modifications of the edge sizes will lead to both extremal runtimes.  In particular, let $e$ and $f$ be small edges such that $\size{e} + 1 < \size{f} \leq k^*$, then 
\begin{align*}
\frac{\size{e} h(\size{e},r) + \size{f} h(\size{f},r)}{\log_2(r)} &= \size{e}2^{\size{e}} + \size{f}2^{\size{f}} \\
&= \paren{\size{e}+1}2^{\size{e}+1} + \paren{\size{f}-1}2^{\size{f}-1}  - \paren{\size{e}+2}2^{\size{e}} + \paren{\size{f}+1}2^{\size{f}-1} \\
&\geq \paren{\size{e}+1}2^{\size{e}+1} + \paren{\size{f}-1}2^{\size{f}-1} \\
&= \frac{\paren{\size{e}+1} h(\size{e}+1,r) + \paren{\size{f}-1} h(\size{f}-1,r)}{\log_2(r)} 
\end{align*}  
As a consequence, the run time of the small edges is minimized when all $m_S$ of the edges have the same size, $\delta_S$, and it is maximized when every small edge has size either $2$ of $k^*$. 

Thus, the maximum total run time is at most 
\[
8\log_2(r) \frac{(r-\delta)m -(r-k^*)m_k}{r-2}+ (k^*)^2 r\log_2(r) m_k +r^{3}\log_2(r) \frac{(\delta - 2)m - (k^*-2)m_k}{r-2}    
\]
where $m_k$ is the number of edges of size $k^*$.  
As the coefficient on $m_k$ in this expression is asymptotically negative, this is maximized when $m_k$ is 0.  That is, when the only edge sizes are either $2$ and $r$, and thus the asymptotic runtime is $\bigOh{r^2\log_2(r)\Vol(H)}.$

To address the minimum run time, we recall from above that the run time is minimized if all the small edges and all the large edge have the same size.  For notational convenience, we define $V_S = \delta_S m_S$ and $V_F = \delta_F m_F$ as the volume of edges which are processed by subset expansion and FFT, respectively. Then, the asymptotic run time is given by
\[ V_S 2^{\frac{V_S}{m_S}}\log_2(r) + m_F \paren{\frac{V_F}{m_F}}^2r\log_2(r) = V_S 2^{\delta_S}\log_2(r) + V_F \delta_F r\log_2(r).\] 

Now suppose that $\nicefrac{V_F}{\Vol{H}} \in \Omega({1})$.  Then, as $\delta_F \geq \delta$, we see that 
\[ \frac{V_F \delta_F r\log_2(r)}{\Vol(H) \delta r\log_2(r)} \in \bigOmega{1}.\]
In particular, this implies that the only case where the run time is $\lilOh{\Vol(H)\delta r\log_2(r)}$ is when $V_S = (1-\lilOh{1}) \Vol(H)$ and $\delta_S = (1-\lilOh{1)}\delta$. As $\delta_S < k^*$, this implies that if $\delta \geq k^*$, then the minimal run time is $\bigOmega{\Vol(H)\delta}$.

  We note that we may view $V_S$ and $V_K$ as functions of $\delta_S$ and $\delta_K$, yielding
\[ f(\delta_S,\delta_F) = \frac{m\delta_K - \Vol(H)}{\delta_K - \delta_S}\delta_S 2^{\delta_s}\log_2(r) + \frac{\Vol(H) - m\delta_S}{\delta_F - \delta_S}\delta_F^2 r\log_2(r).\]

It is a straightforward, albeit tedious exercise, to show that, taken as a function of $\delta_F$, $f$ is minimized when $\delta_F$ takes the value 
\[ \delta_F^* = \delta_S + \delta_S\sqrt{1 - \frac{2^{\delta_S}}{\delta_S r}} \leq 2\delta_S.\]
In particular, this implies if $\delta_S < \frac{1}{2}k^*$  (equivalently, if $\delta < k^*$) then the minimum run time is 
\[\bigTheta{ \Vol(H) 2^{\delta}\log_2(r)}.\]

\end{proof}

We note the gap between upper and lower runtime bounds for \hyperref[alg:TTSV1G]{\ttsvG} and \hyperref[alg:TTSV2G]{\ttsvG[2]} are on the order of $\nicefrac{r}{\delta}$ and $\paren{\nicefrac{r}{\delta}}^2$, respectively. This is, in a sense, the best possible gap as it represents the ratio between the maximum and minimum second and third moment, respectively, of edge sizes for a fixed average size.  In practice, these moments govern the runtime for these algorithms up to a multiplicative function based on the max edge size $r$.

\section{The hypergraph adjacency tensor in practice}\label{sec:applications}

We explain how to use our \ttsv\ algorithms to perform fundamental hypergraph analyses, which we then apply to the datasets listed and summarized in Table \ref{tab:allexp1}. In particular, we focus on algorithms for computing tensor eigenvector centralities, as well as a CP-decomposition approach for performing tensor-based hypergraph clustering. Our study here is meant to be illustrative rather than exhaustive: \ttsv\ algorithms find far-ranging application in tensor computation beyond these two tasks. While our goal is not to argue that the specific centrality and clustering algorithms we propose are the best of their type, we show they provide complementary information to graph reduction approaches performed on the clique expansion, and provide a concrete example to illustrate they detect subtle, higher-order structure in hypergraphs that a wide-variety of popular, existing matrix-based methods provably cannot. Taken together, this suggests the tensor-approaches enabled by our algorithms are flexible, tractable, and worthwhile avenues for nonuniform hypergraph data analysis. Before proceeding, we compare the empirical run times of the aforementioned \ttsv\ algorithm variants. All experiments were run on a single core of a MacBook Pro with an M1 Max processor and 32GB of RAM.

\begin{minipage}{\textwidth}
\centering
\begin{minipage}[b]{.6\textwidth}
% \begin{minipage}{0.5\linewidth}
  \centering
  \scalebox{0.75}
  {
  \begin{tabular}{l     l l l l l l}
    \toprule
    %\!\!\!\!\!\!\!
    \emph{Dataset}              & Ref. & $\lvert V \rvert$ & $ \lvert E \rvert$ & $\delta$ & $\Vol(H)$ & $r$\\
    \midrule
    % contact-high-school~\cite{chodrow2021hypergraph,Mastrandrea-2015-contact}   & 327                & 7.8K          & 2.3       & 17.9K & 5\\
    % contact-primary-school~\cite{Stehl-2011-contact}                            & 242                & 12.7K         & 2.4       & 30.5K & 5\\
    \texttt{mathoverflow} & \cite{Veldt-2020-local}                                & 73.8K              & 5.4K         & 24.2      & 132K  & 1.8K\\
    \quad {\footnotesize \it filtered} & & \footnotesize 39.8K              & \footnotesize 5.2K          & \footnotesize 11.3       & \footnotesize 58.8K & \footnotesize 100 \\
    \texttt{MAG-10} & \cite{amburg2020clustering,Sinha-2015-MAG}                           & 80.2K              & 51.9K         & 3.5      & 181K & 25\\
    \texttt{DAWN} & \cite{amburg2020clustering}                                            & 2.1K               & 87.1K         & 3.9      & 343K  & 22\\
    \texttt{cooking} & \cite{amburg2020clustering}                                         & 6.7K               & 39.8K         & 10.8     & 429K  & 65\\
    \texttt{walmart-trips} & \cite{amburg2020clustering}                                   & 88.9K              & 69.9K         & 6.6      & 461K  & 25\\
    \texttt{trivago-clicks}& \cite{chodrow2021hypergraph}                                 & 173K               & 233K          & 3.1      & 726K  & 86\\
    \texttt{stackoverflow} & \cite{Veldt-2020-local}                               & 15.2M              & 1.1M &  23.7& 26.1M & 61.3K \\
    \quad {\footnotesize \it filtered} & & \footnotesize 7.8M               & \footnotesize  1.0M           & \footnotesize  10.1       & \footnotesize 10M & \footnotesize  76\\
    \texttt{amazon-reviews}&\cite{ni2019justifying}                                      & 2.3M               & 4.3M          &  17.1    & 73.1M & 9.4K \\
    \quad {\footnotesize \it filtered} & & \footnotesize 2.2M & \footnotesize 3.7M & \footnotesize 9.6 & \footnotesize 35.6M & \footnotesize 27 \\
    \bottomrule
  \end{tabular}
  }
    \captionof{table}{Summary statistics of datasets.} \label{tab:allexp1}
%   \end{minipage}
\end{minipage}
\begin{minipage}[b]{135pt}
    \centering
    % \begin{minipage}[c]{width=0.4\textwidth}
    \includegraphics[trim = 10 15 5 10, clip, width=130pt]{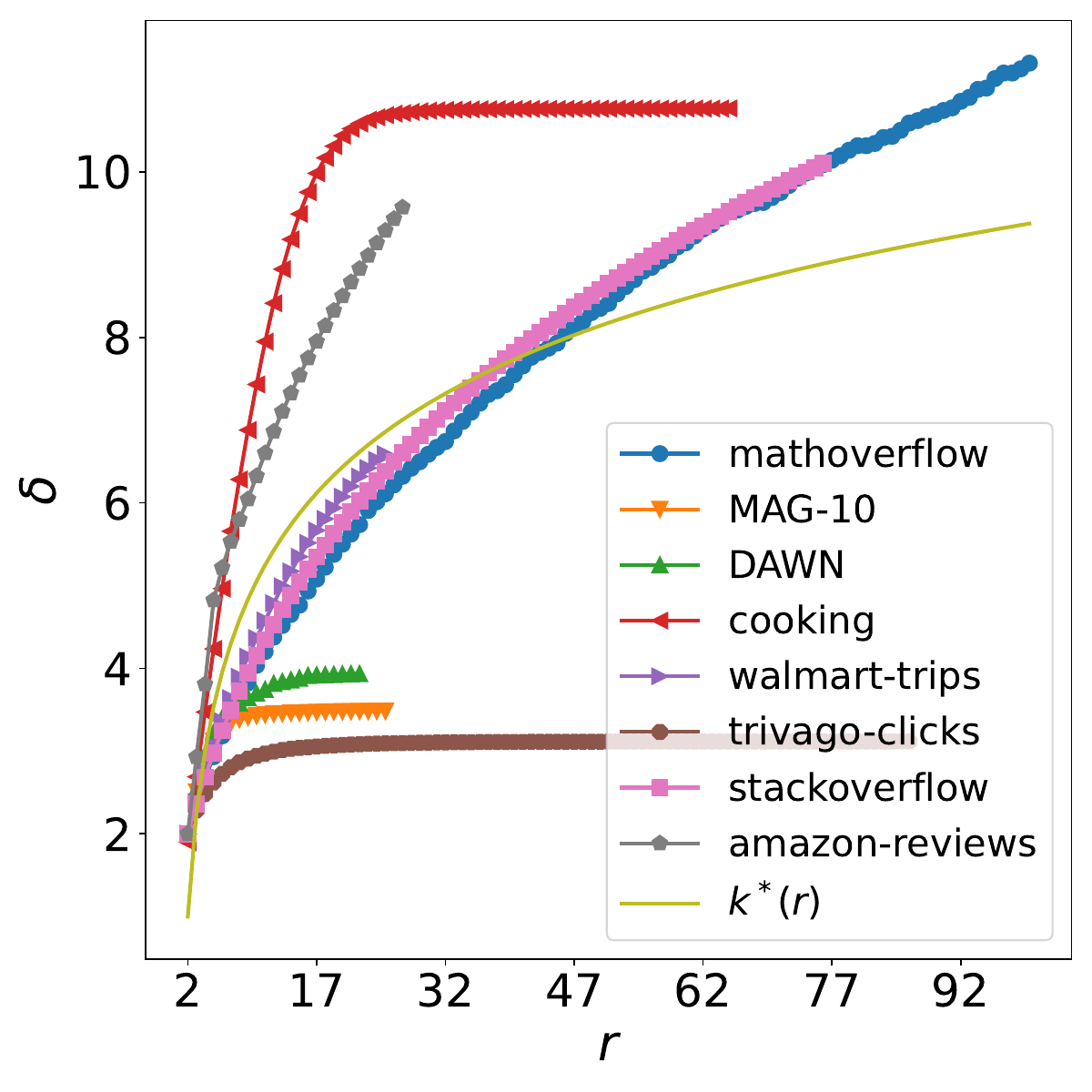}
    \captionof{figure}{$\delta$ vs. $r$}
    \label{fig:my_label}
    % \end{minipage}
\end{minipage}
%\hspace{15pt}
\end{minipage}
% Want the product of textwidth parameter and line width parameter to be .28

\subsection{Timing experiments} 
    % \footnotetext{These datasets are filtered versions of the original with a maximum hyperedge size of either a) 100, to preclude floating-point errors, or b) the hyperedge size after which \hyperref[alg:TTSV1G]{\ttsvG} times out after an hour.} 
We compare the empirical runtimes of \hyperref[alg:TTSV1G]{\ttsvG} and \hyperref[alg:TTSV1U]{\ttsvU} against two baselines: {\it explicit}, in which $\A$ is constructed explicitly and \ttsv[1] is computed from definition, and {\it ordered}, which is identical to \hyperref[alg:TTSV1U]{\ttsvU} except iterates over full blowup sets $\beta(e)$ rather than $\kappa(e)$. 
To study scaling in hypergraph rank, we increase $r$ from 2 up to a maximum of 100 or the largest hyperedge size\footnote{See~\cref{sec:numerical} for further discussion of numerical considerations.}, or until the execution timeout of one hour. This procedure corresponds to the ``less than or equal to'' (LEQ) filtering discussed in~\cite{landry2023filtering}. With this scheme, we are able to process \texttt{MAG-10}, \texttt{DAWN}, \texttt{walmart-trips}, \texttt{cooking}, and \texttt{trivago-clicks} in their unfiltered entirety, \texttt{mathoverflow} until $r=100$, to guard against floating-point errors, and \texttt{stackoverflow} and \texttt{amazon-reviews} until timeout at $r=76$ and $r=27$, respectively. 

Figure \ref{fig:times} presents the TTSV1 results for all datasets in Table \ref{tab:allexp1}. Similar results for TTSV2 are reported in Figure \ref{fig:ttsv2-times} of the supplementary material. Missing times indicate the algorithm either timed out after an hour or forced an out-of-memory error. For all datasets, the explicit algorithm becomes intractable after $r>3$. The separation between the green and blue lines attest to the significant speedup from considering only unordered blowups. Furthermore, \hyperref[alg:TTSV1G]{\ttsvG} provides additional speedup many orders of magnitude over \hyperref[alg:TTSV1U]{\ttsvU}, even for modest values of $r$. For example, in \texttt{mathoverflow}, the speedup at $r=9$ is about an order of magnitude, and increases to about three orders of magnitude by $r=18.$\footnote{Our implementation utilizes two heuristic approaches which, while not affecting the asymptotic analysis, result in significant performance gains: for sufficiently small $r$ we use a direct convolution approach instead of the FFT~\cite{ismail2009convolution,2020SciPy-NMeth}, and for edges of size $r$ use a direct approach identical to the uniform hypergraph case~\cite{Benson-2018-core}.}

\begin{figure}
    \centering
    \includegraphics[width=0.99\linewidth]{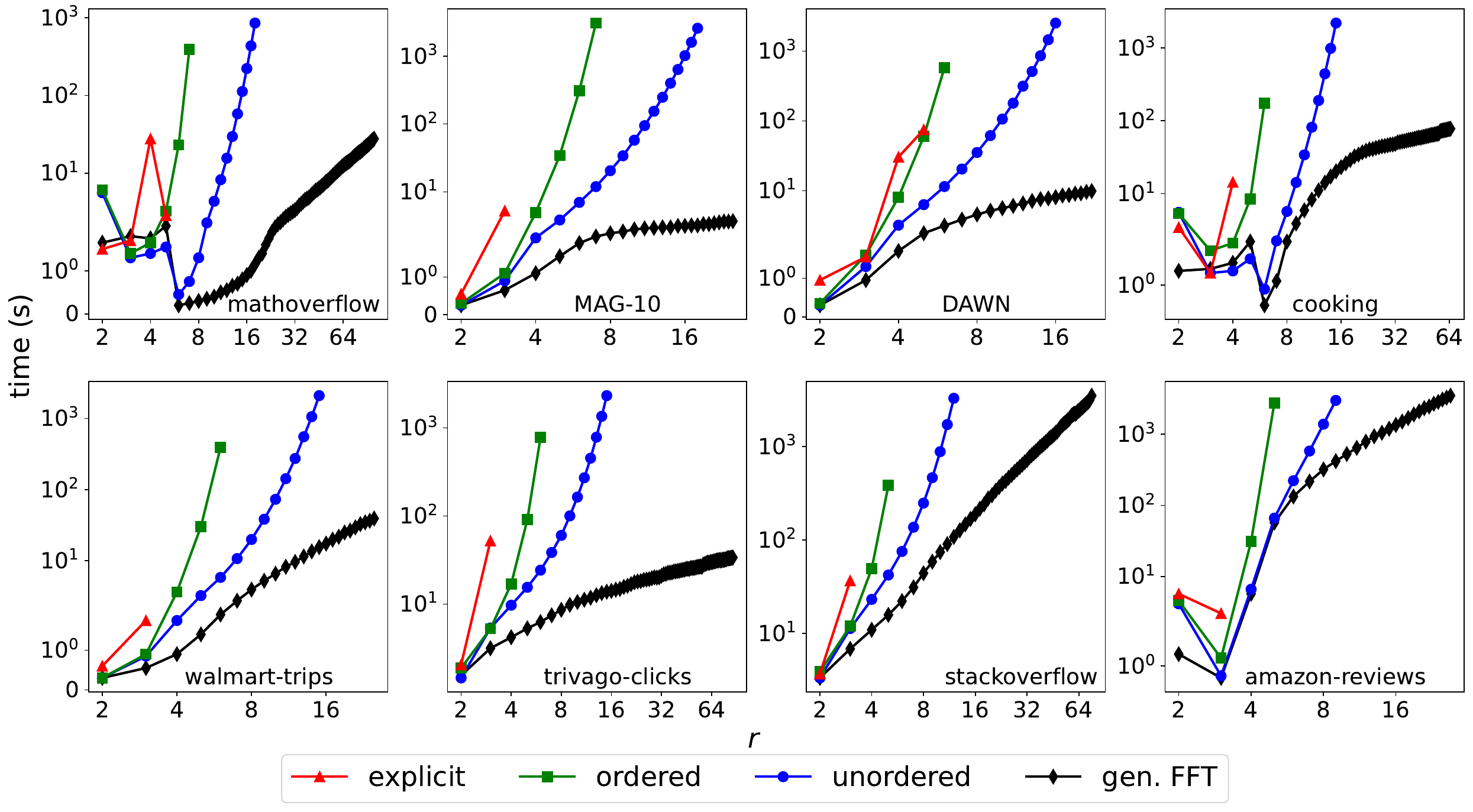}
    \caption{
    Runtimes of TTSV1 algorithms. 
    %  Runtimes of Algorithms~\ref{alg:TTSV1gene} (left) and \ref{alg:TTSV2gene} (right).
    }
    \label{fig:times}
\end{figure}

Having established the superiority of \hyperref[alg:TTSV1U]{\ttsvU}, we study its performance on the hyperedge level as a function of hyperedge size and $r$. We select \texttt{DAWN} and \texttt{cooking}, as their hyperedge distributions differ substantially.
Figure \ref{fig:ttsv1_by_hyperedge} presents timing results for \hyperref[alg:TTSV1G]{\ttsvG}, as well as the hyperedge size distributions. 
For \texttt{DAWN}, the bulk of compute time for most values of $r$ is spent on edges with sizes $3\leq \lvert e\rvert \leq 6$, yet the per-hyperedge runtime for edges in this range is relatively fast.
These observations are reconciled by the fact that the edge degree distribution is right-skewed. For \texttt{cooking}, the edge size distribution is even more right-skewed, but also has a mode at $\lvert e\rvert = 9$, exceeding that of \texttt{DAWN}, which is 3. As a result, most of the runtime is concentrated around edges of size $8\leq \lvert e\rvert \leq 17$.
However, the per-edge runtime is still dominated by the larger edges for higher values of $r$. Note many edge sizes beyond 36 are not present in the dataset, yielding the white striations in the runtime plots.

\begin{figure}
    \centering
    \includegraphics[width = 0.31\linewidth]{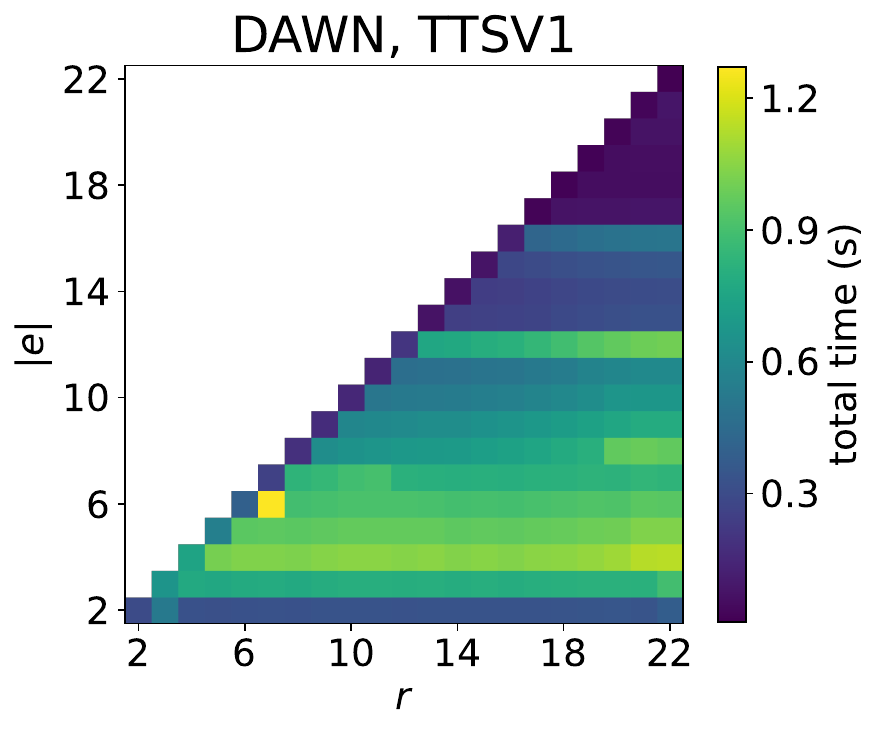}
    \includegraphics[width = 0.33\linewidth]{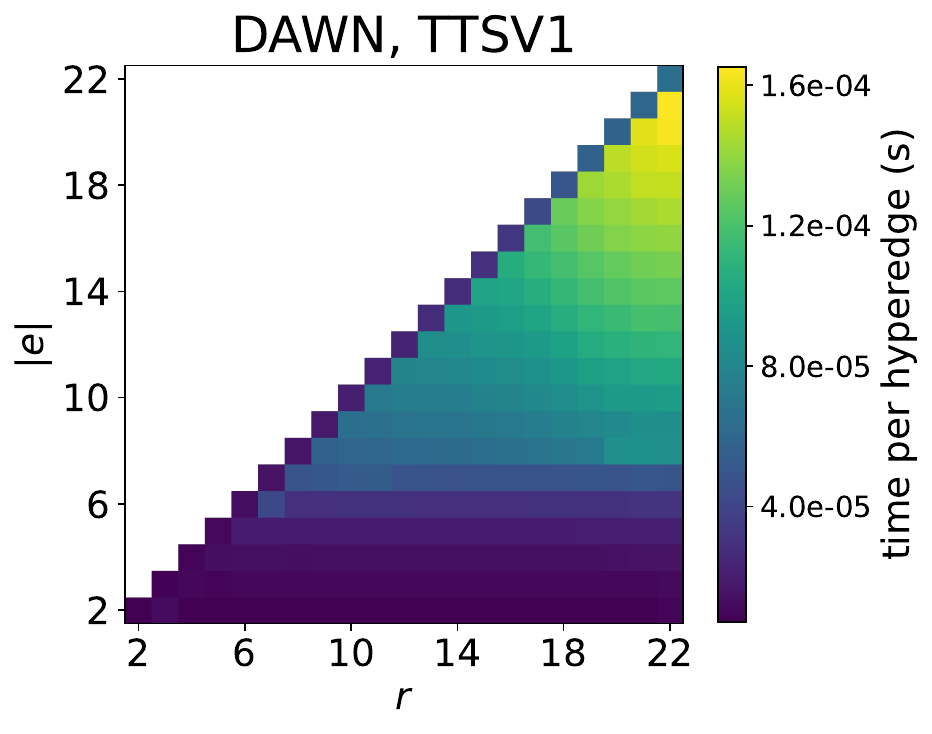}
    \includegraphics[width = 0.25\linewidth]{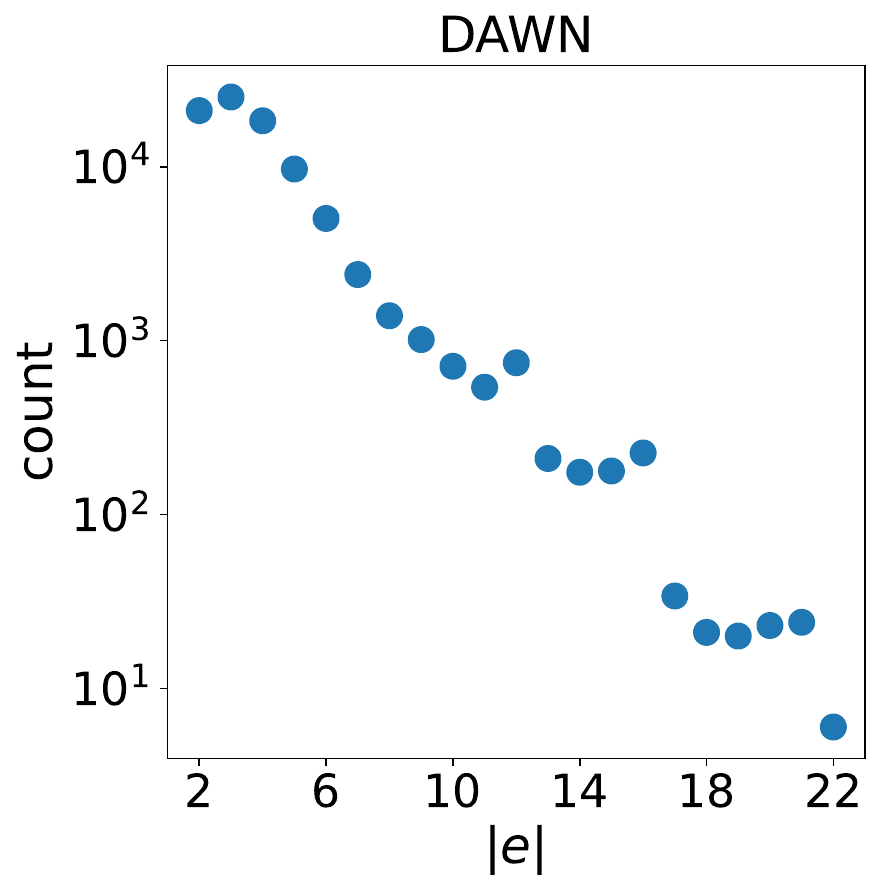}
    \includegraphics[width = 0.31\linewidth]{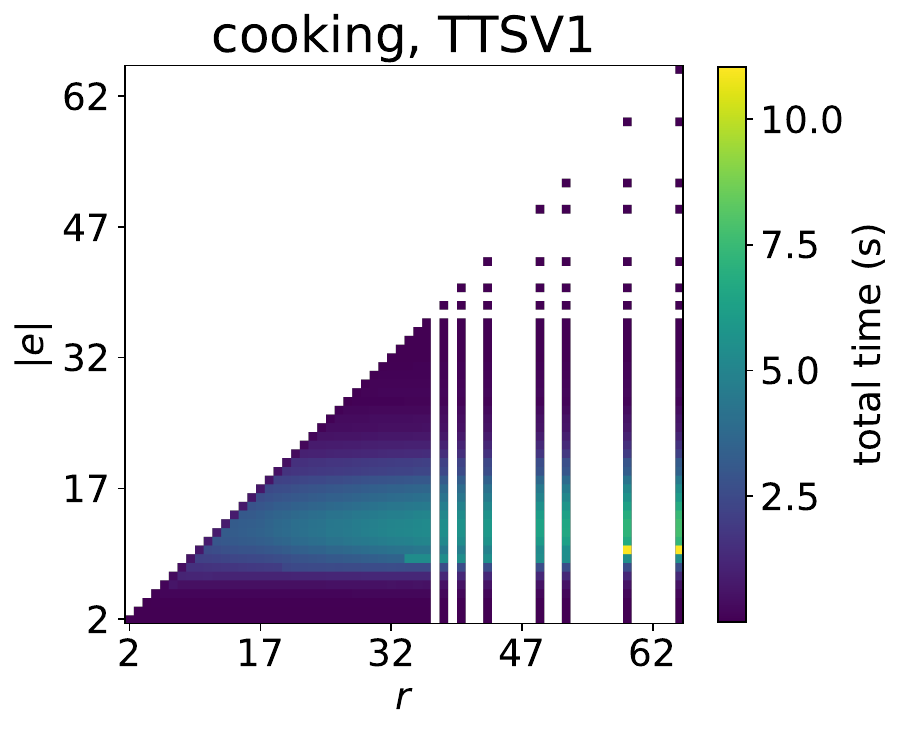}
    \includegraphics[width = 0.33\linewidth]{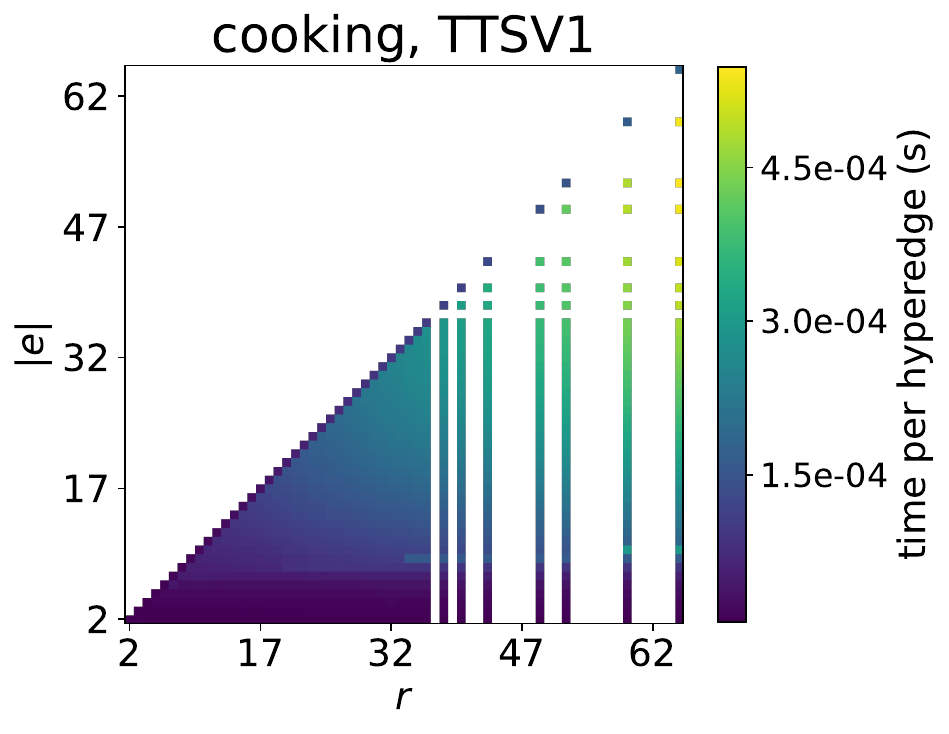}
    \includegraphics[width = 0.25\linewidth]{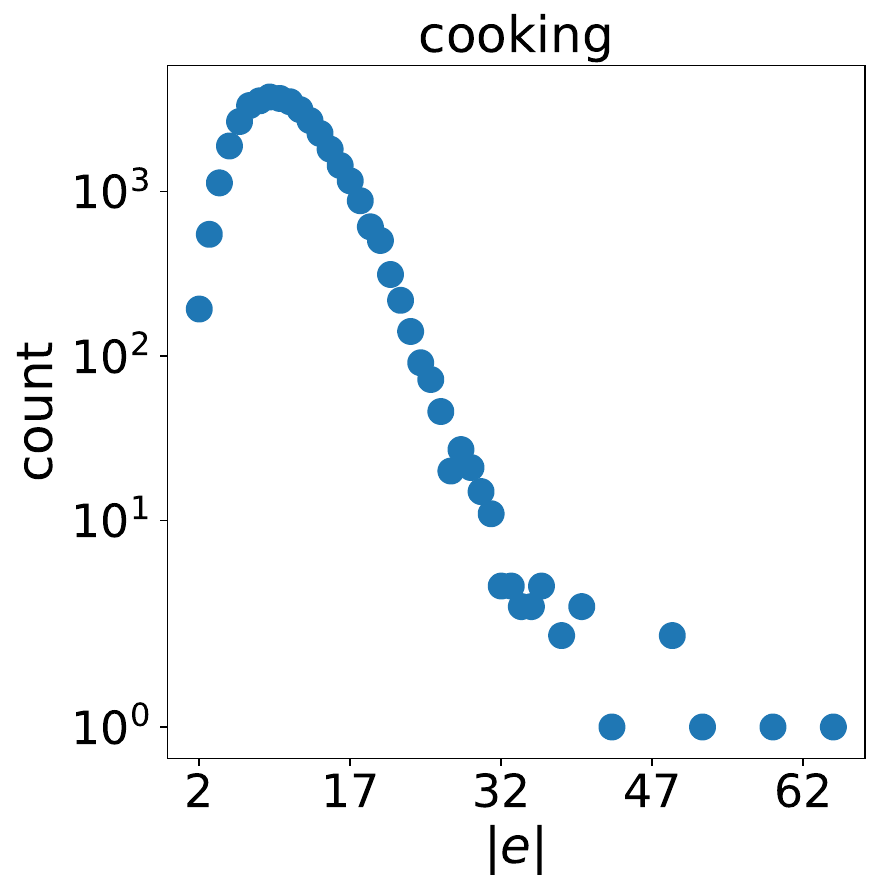}
    \caption{Total (left) and per-edge (middle) runtimes for hyperedge size $\lvert e\rvert$ as a function of rank $r$, and hyperedge size distributions (right) for \texttt{DAWN} and \texttt{cooking}.
    }
    \label{fig:ttsv1_by_hyperedge}
\end{figure}

\subsection{Centrality}
%xhdr{A nonlinear tensor ``eigenvector'' centrality for nonuniform hypergraphs}
Our TTSV algorithms may be utilized to compute nonlinear hypergraph centralities. Here, we discuss several such centrality measures in the nonuniform setting, apply the relevant Perron-Frobenius theory guaranteeing their existence, and then compute them\footnote{Our centrality code is available at \url{https://github.com/pnnl/GENTTSV}.} on data and illustrative toy examples to show they yield meaningful, complementary information to existing centrality measures. We focus on the tensor-based $Z$-eigenvector centrality (ZEC) and $H$-eigenvector centrality (HEC) introduced by Benson \cite{benson2019three} for \textit{uniform} hypergraphs. Here, the centrality $c_u$ of node $u$ depends on the sum of {products} of its neighbors' centralities. 
Using the formalism of hyperedge blowups, we note this same concept may be applied to non-uniform hypergraphs where ZEC and HEC satisfy 

%Extending these to the nonuniform setting,  ZEC and HEC satisfy
\begin{align*}
\sum_{e\in E(u)}\sum_{x \in \beta(e) }\left[\tfrac{1}{c_u}\prod_{v\in x}c_v\right] &= \lambda c_u,\\
\sum_{e\in E(u)}\sum_{x \in \beta(e) }\left[\tfrac{1}{c_u}\prod_{v\in x}c_v\right] &= \lambda c_u^{r-1},
\end{align*}
respectively. Note these differ in that the second preserves dimensionality on both sides of the equation. Both of these problems can be cast as eigenvector problems whose existence, as we soon show, is guaranteed if the hypergraph is connected. 

\begin{definition}[$Z$- and $H$-eigenvector centrality, (ZEC and HEC)]
Let $H$ be a connected, rank $r$ hypergraph with adjacency tensor $\A$. The $Z$- and $H$-eigenvector centrality of $H$ is given by a vector $\Vc{c}$ with $||\Vc{c}||_1=1$ satisfying
\begin{align*}
    \A\Vc{c}^{r-1} &= \lambda \Vc{c}, \\
    \A\Vc{c}^{r-1} &= \lambda \Vc{c}^{[r-1]},
\end{align*}
respectively, for some eigenvalue $\lambda>0$.
\end{definition}

While the basic pattern of nonlinearity afforded by these methods is the same in the uniform and nonuniform hypergraph settings, the intuition behind these measures differs slightly. For ZEC in the uniform setting, the centrality of a node is proportional to the sum of products of the centralities of all nodes in edges that contain it. In the nonuniform setting, however, the importance of a node derives from multiplying the centralities of other nodes within all possible {\it blowups} of the containing hyperedges. Despite this apparent ``uneveness" in how node importance is computed in the product, considering all possible blowups ensures there is no artificial bias to any node within the hyperedge. HEC proceeds similarly, except the polynomials on both sides of the equation have the same degree for every term -- a desirable property whenever dimensionality preservation is important, such as in physics-based applications.

The Perron-Frobenius theory guarantees the existence of the $Z$- and $H$-eigenvectors for the nonuniform hypergraph adjacency tensor $\A$ as follows: recall, from~\cite{qi2017tensor}, that an order $k$, $n$-dimensional tensor $\T{X}$ is {\it irreducible} if the associated directed graph $(V,E)$ with $V=\{1,\dots,n\}$ and
\[
E=\{(i,j) : \mbox{ there exists } I=\{i_2,\dots,i_n \} \mbox{ with } j \in I \mbox{ and } \T{X}_{ii_2\dots i_k} \neq 0\},
\]
is strongly connected. Via the symmetry of $\A$, it easily follows that its associated digraph is precisely the hypergraph's clique expansion, in which every directed edge $(i,j)$ is reciprocated by an edge $(j,i)$. Consequently, defining a connected hypergraph as one whose clique expansion is connected, the Perron-Frobenius theorem for nonnegative tensors \cite[Theorem 3.11, p. 50]{qi2017tensor} is applicable to connected, nonuniform hypergraphs as follows.

\begin{theorem}[Perron-Frobenius theorem for the hypergraph adjacency tensor \cite{qi2017tensor}]
If $H$ is a connected, rank $r$ hypergraph with adjacency tensor $\A$, then there exists
\begin{itemize}
    \item $Z$-eigenpair $\lambda>0,\Vc{x}>0$ satisfying $\A \Vc{x}^{r-1}=\lambda\Vc{x}$, and
    \item $H$-eigenpair $\lambda>0,\Vc{x}>0$ satisfying $\A\Vc{x}^{r-1}=\lambda\Vc{x}^{[r-1]}$ where $\lambda$ is the largest $H$-eigenvalue of $\A$ and $\Vc{x}$ is unique up to scaling.
\end{itemize}
\end{theorem}

\begin{center}
\begin{minipage}[t]{0.475\linewidth}
  \vspace{0pt}  
  \phantomsection\label{alg:BG}
\begin{algorithm}[H]\small
  \LinesNotNumbered 
% \SetLine
\SetKwRepeat{Repeat}{repeat}{until}
\KwData{$n$-vertex, rank $r$ hypergraph $H$, tolerance $\tau$, step size $h$}
\KwResult{Z-eigenvector centrality, $\Vc{x}$}
$\Vc{y}=\frac{1}{n}\cdot \Vc{1}$\;

\Repeat{$\frac{\max \Vc{g} - \min \Vc{g}}{\min \Vc{g}}<\tau$}{
$\Mx{Y}=\ttsv[2](H,\Vc{y})$\;

$\Vc{d} = \mbox{dominant eigenvector of } \Mx{Y}$\;

$\Vc{x} = \Vc{y}+h\cdot(\Vc{d}-\Vc{y})$\;

$\Vc{g}=\Vc{x}\oslash \Vc{y}$\;

$\Vc{y} = \Vc{x}$\;
}
\Return{$\Vc{x}$}
\caption*{\textsc{BG}:  $Z$-eigenvector centrality}\label{alg:zeig}
\end{algorithm} 
\end{minipage}\hfil \hfil
\begin{minipage}[t]{0.475\linewidth}
 \vspace{0pt}
 \phantomsection\label{alg:NQZ}
 \begin{algorithm}[H]\small
  \LinesNotNumbered 
% \SetLine
\KwData{$n$-vertex, rank $r$ hypergraph $H$, tolerance $\tau$}
\KwResult{H-eigenvector centrality, \Vc{x}}
$\Vc{y}=\frac{1}{n} \cdot \Vc{1}$\;

$\Vc{z}=\ttsv[1](H,\Vc{y})$\;

\Repeat{$(\lambda_{\max} -\lambda_{\min})/\lambda_{\min} < \tau$}
{
$\Vc{x}=\Vc{z}^{\left[\frac{1}{r-1}\right]} / \norm{\Vc{z}^{\left[\frac{1}{r-1}\right]}}_1$\;

$\Vc{z}=\ttsv[1](H,\Vc{x})$\;

$\lambda_{\min} = \min{(\Vc{z} \oslash \Vc{x}^{\left[r-1\right]})}$\;

$\lambda_{\max} = \max{(\Vc{z} \oslash \Vc{x}^{\left[r-1\right]})}$\;
}
\Return{$\Vc{x}$}
\caption*{\textsc{NQZ}: $H$-eigenvector centrality}\label{alg:heig}
\end{algorithm}
\end{minipage}
%\caption{Left to right: Algorithms for hypergraph $Z$ and $H$-eigenvector centrality.}\label{fig:centralities}
\end{center}

Observe here that uniqueness up to scaling is only guaranteed for $H$-eigenpairs, the theory of which is generally stronger than for the $Z$-counterparts. 
Having established the necessary theory, we now provide implementation details.  ZEC is computed using a dynamical system approach proposed by Benson and Gleich in~\cite{benson2019computing}.  Specifically, we have 
\[
[\A\Vc{x}^{r-2}]\Vc{x}=\lambda \Vc{x} \iff \A\Vc{x}^{r-1}=\Vc{x},
\]
which is the same as requiring that the system 
\[
\frac{dx}{dt}=\Lambda(\A\Vc{x}^{r-2})-\Vc{x}
\]
has a steady state solution, where $\Lambda$ maps a matrix to its dominant eigenvector. Consequently, any forward integration scheme can be applied together with an eigenvector solver to get a tensor eigenvector through a nonlinear matrix eigenvector problem.  \hyperref[alg:BG]{\textsc{BG}} presents a concrete instantiation of the Benson-Gleich approach which utilizes \ttsv[2] as a subroutine.
 For HEC, we use a power-iteration like method along the lines of Ng, Qi, and Zhou~\cite{ng2010finding}, presented in \hyperref[alg:NQZ]{\textsc{NQZ}}, which relies on the calculation of \ttsv[1] as a subroutine.

\begin{figure}
    \centering
    \includegraphics[width=0.45\linewidth]{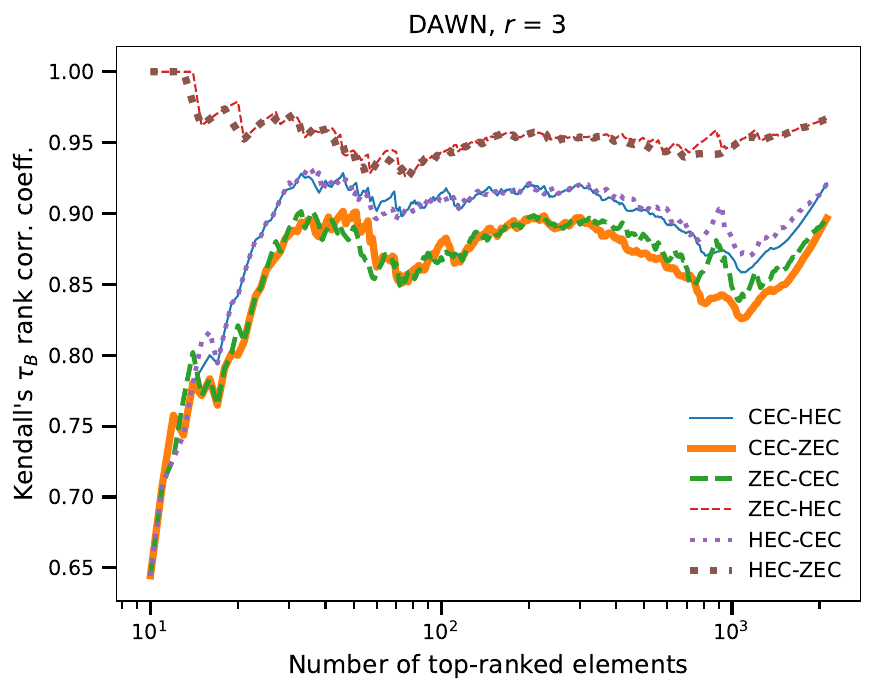} 
    \includegraphics[width=0.45\linewidth]{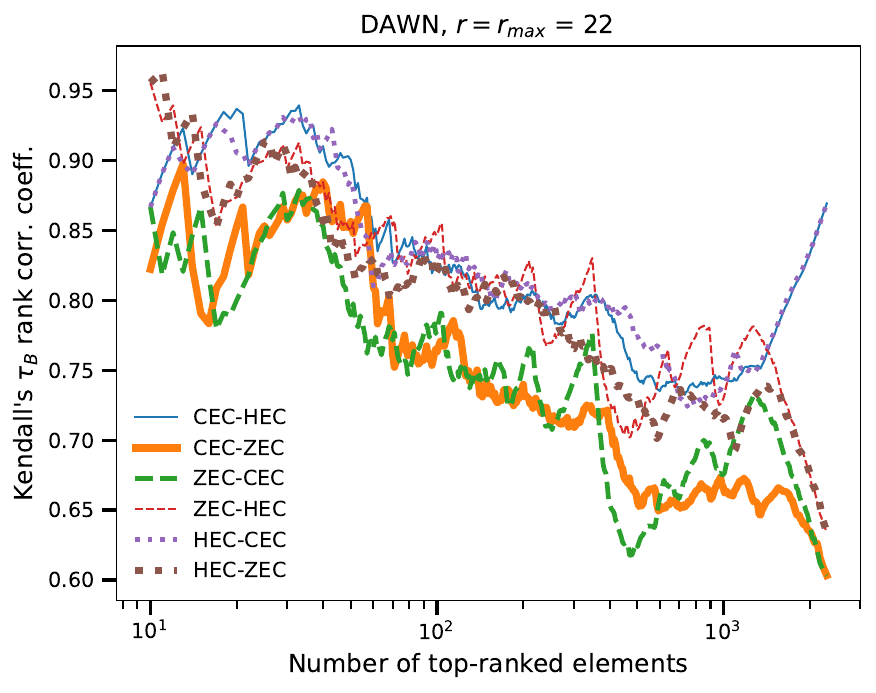}
    \includegraphics[width=0.45\linewidth]{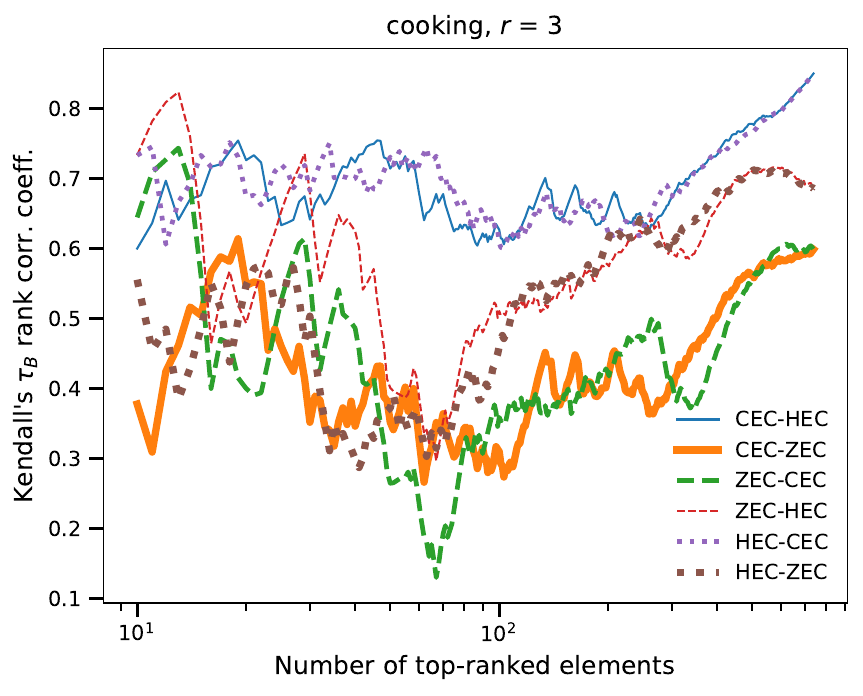}
    \includegraphics[width=0.45\linewidth]{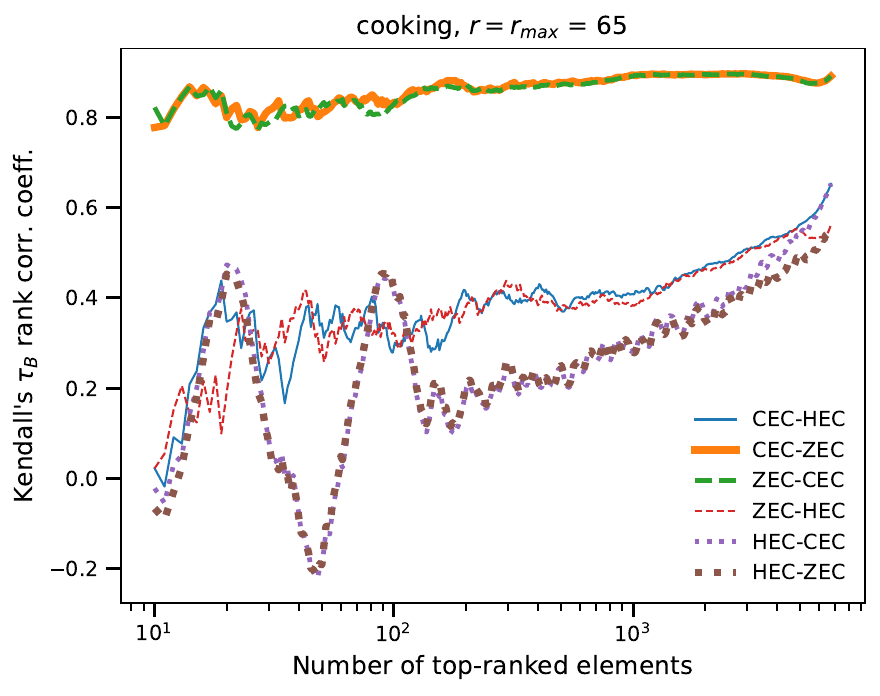} \\
    \caption{Kendall $\tau_B$ rank correlation coefficient of ZEC, HEC, and CEC for the top $k$  nodes. %\sga{Can this get the same treatment as the other: legend at the bottom, on y-axis label on left, on x-axis label on bottom? Thanks!}
    }
    \label{fig:kendall}
\end{figure}

Applying these algorithms to \texttt{DAWN} and \texttt{cooking}, we now address three questions: (1) do our nonuniform hypergraph centrality measures provide different centrality scores in practice from existing methods, (2) how do these rankings change as we vary the maximum hyperedge size $r$, and (3) can these measures detect higher-order structure in hypergraphs that is inexpressible in graphs and therefore undetectable by clique expansion and associated approaches? Section~\ref{sec:case-study} presents additional case-study for \texttt{cooking} and \texttt{DAWN}.

\subsubsection{Tensor centralities provide complementary information}\label{sec:different}
We compare ZEC and HEC scores against each other, and against clique expansion centrality (CEC): the dominant eigenvector of the weighted clique expansion adjacency matrix \cite{benson2019three}. We compare the ordinal rankings induced by ZEC, HEC, and CEC by computing Kendall's $\tau_B$ rank correlation coefficient among the top $k$ ranked vertices. Figure \ref{fig:kendall} presents the results for \texttt{DAWN} and \texttt{cooking} for both the full, unfiltered data ($r=22$ and $65$, respectively), as well as for the $r=3$ filtering. We observe CEC rankings are relatively uncorrelated with those of HEC and ZEC, and that no pair of measures exhibits a consistent level of correlation among the top $k$ ranked vertices. When compared against each other, ZEC and HEC are either weakly correlated or uncorrelated, suggesting they provide different information in practice. Lastly, the differences between $r=3$ and the unfiltered data suggest these correlations are sensitive to filtering.

\begin{figure}
    \centering
     \includegraphics[width=0.3\linewidth]{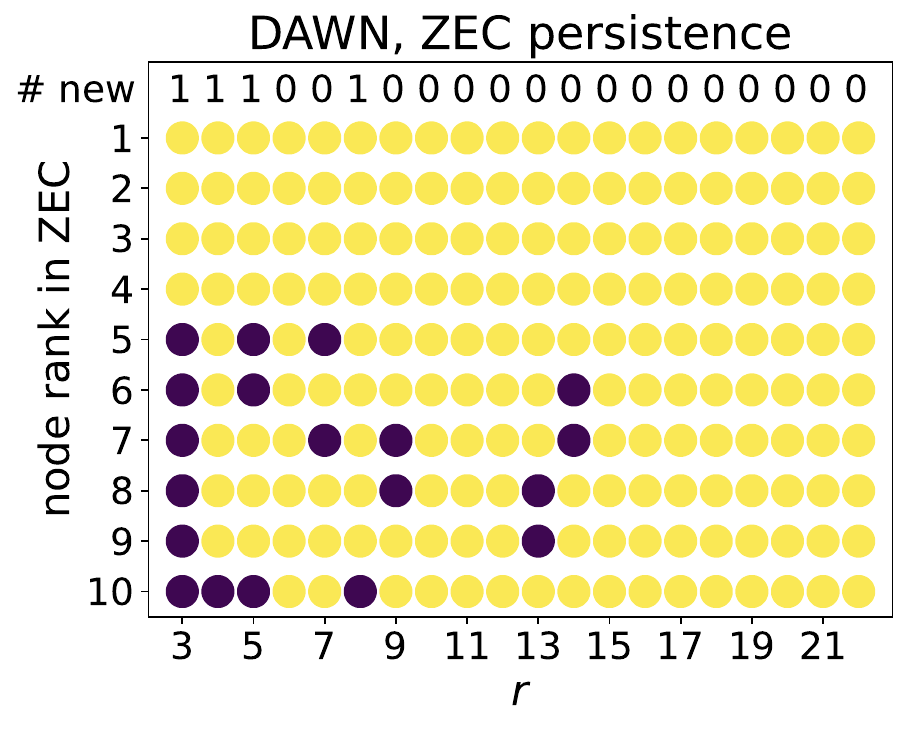}\includegraphics[width=0.3\linewidth]{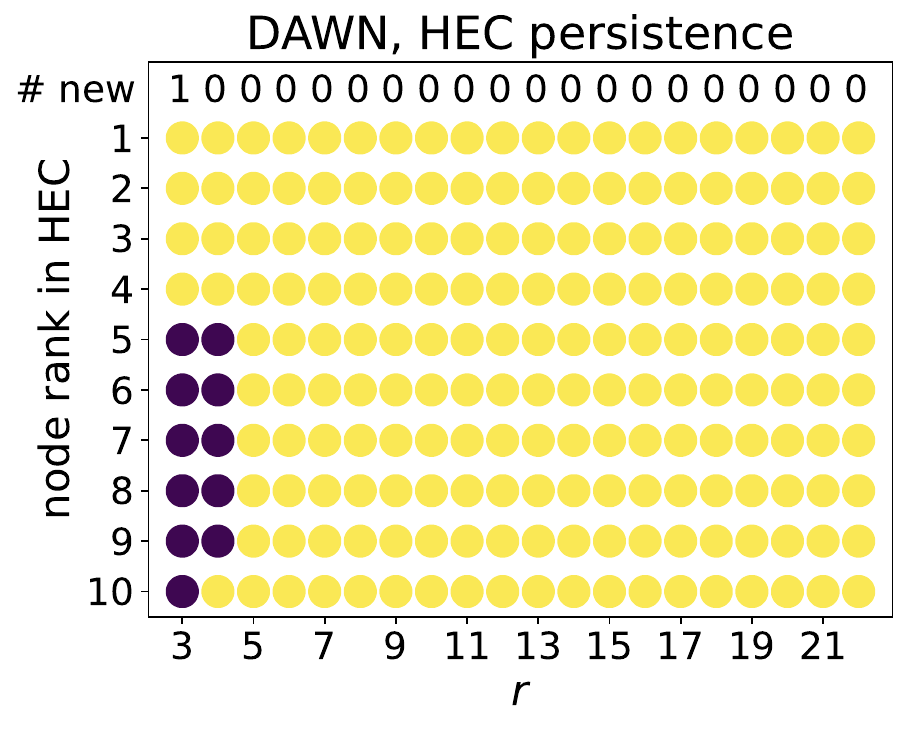}\includegraphics[width=0.3\linewidth]{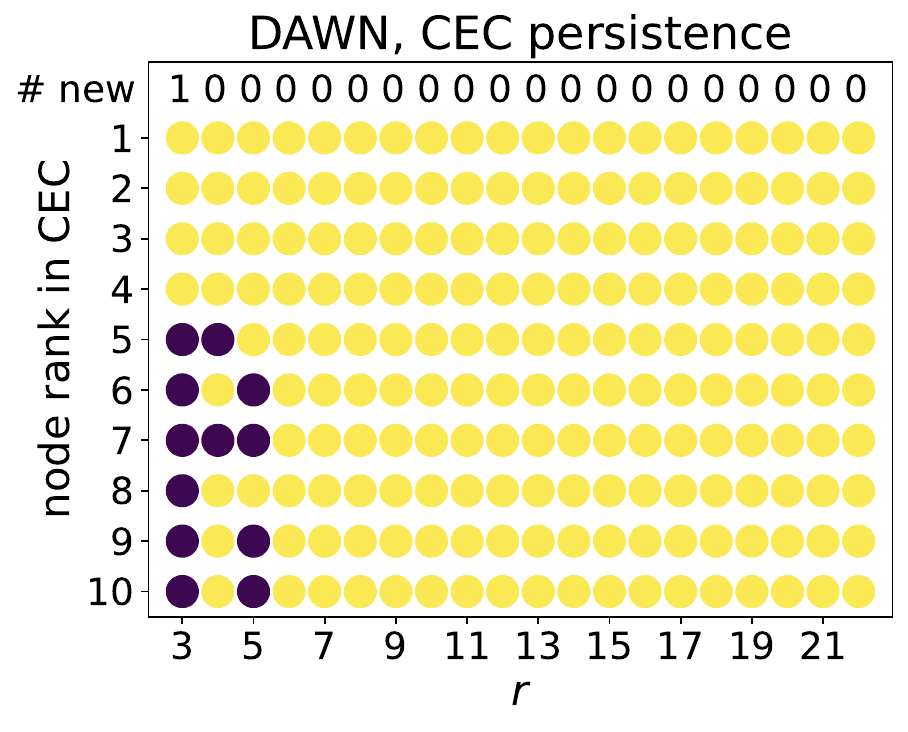}
    \includegraphics[width=0.3\linewidth]{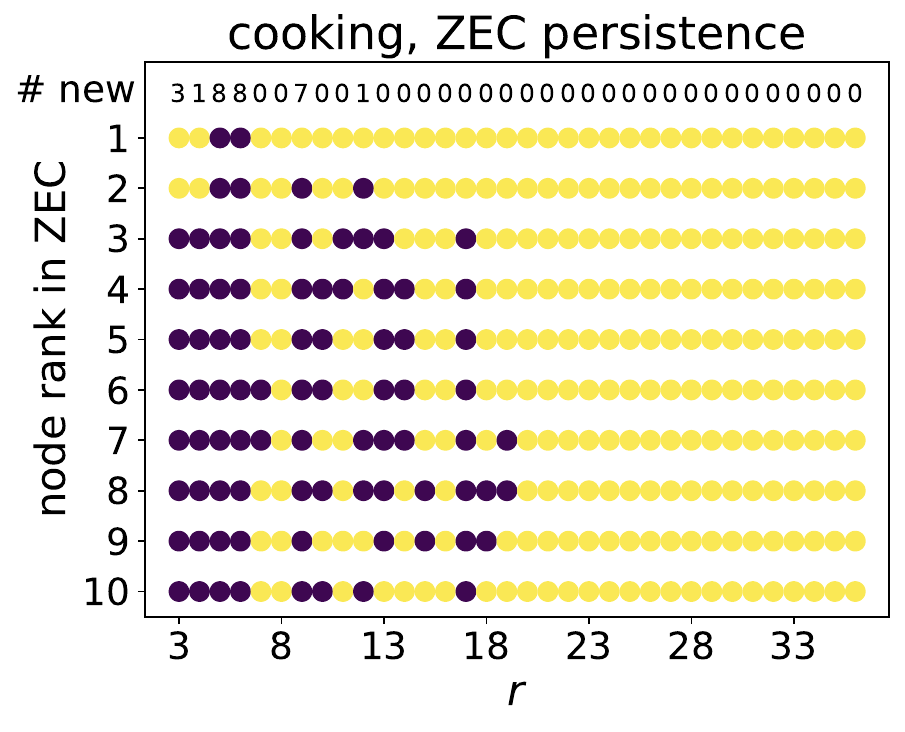}\includegraphics[width=0.3\linewidth]{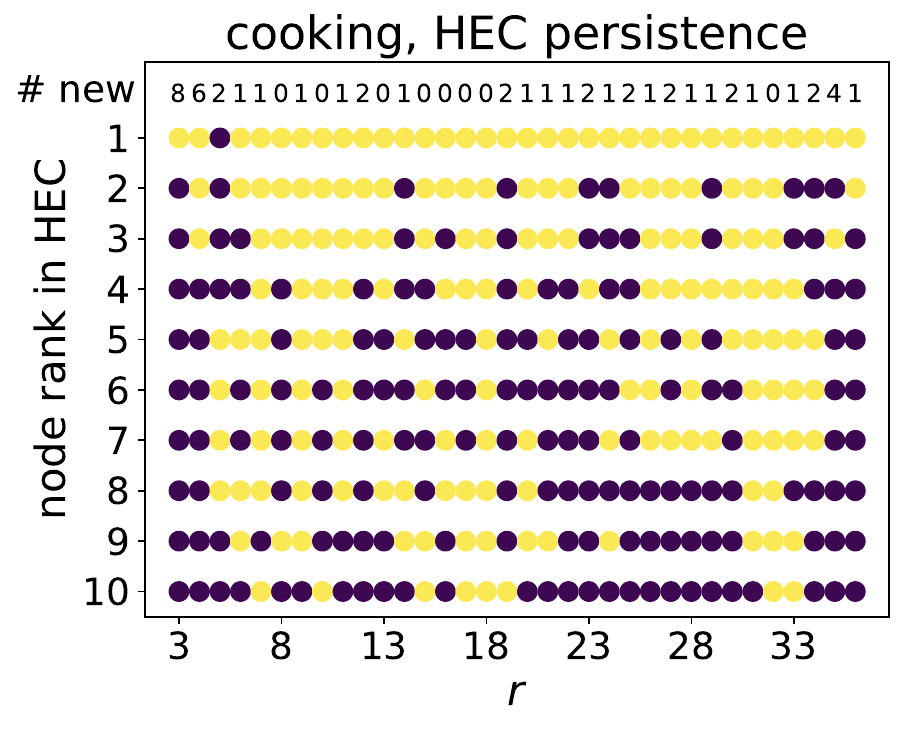}\includegraphics[width=0.3\linewidth]{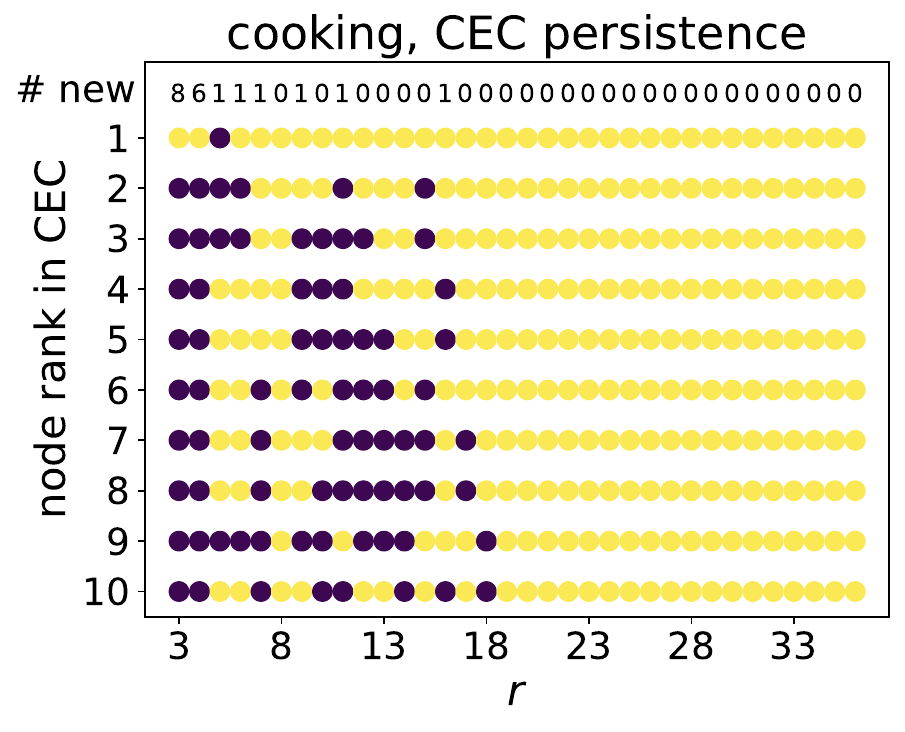}
    \caption{Node rank persistence of the top 10 nodes under ZEC, HEC, and CEC for \texttt{DAWN} and \texttt{cooking}. Purple indicates a rank change from $r-1$ and yellow indicates no change. }
    \label{fig:persistence}
\end{figure}

\subsubsection{Persistence in tensor centrality}\label{sec:persistence}
Next, we study the persistence of rankings induced by ZEC, HEC, and CEC. In particular, we perform an LEQ filtering sweep for $r$ ranging from 2 to the maximum hyperedge size, and record whether the rank of each of the top 10 nodes changes as we increase $r$. Figure \ref{fig:persistence} presents the results for \texttt{DAWN} and \texttt{cooking}, where purple indicates rank change from $r-1$, yellow indicates no change, and the top row lists the number of new nodes in the top 10. For \texttt{DAWN}, CEC and HEC rankings quickly stabilize, showing no changes after $r=5$ and $4$, respectively. In contrast, ZEC rankings stabilize more slowly at $r=14$. For all three centrality measures, however, stabilization occurs before the maximum hyperedge size at $r=22$, echoing the claim in~\cite{benson2019three} that higher-order information is sometimes well-captured by hyperedges that are ``medium" to ``small" relative to the largest hyperedge. Consequently, for larger data (such as \texttt{mathoverflow}) where the maximum hyperedge size is prohibitively large for our algorithms, analyses may still be satisfactorily performed on a filtering to smaller hyperedges. 
This is, of course, dataset and question-dependent, as evidenced by the persistence results for \texttt{cooking}: here, HEC rankings continue to show instability across larger $r$ whereas ZEC and CEC rankings both stabilize around $r=18$. %Here, the HEC tensor centrality exhibits strong differences with both
This highlights how these tensor centralities can differ from matrix-analogs, as well as from each other.

\begin{figure}
    \centering
    \subfloat[Hypergraph $S$]{%
        \includegraphics[width=0.45\linewidth]{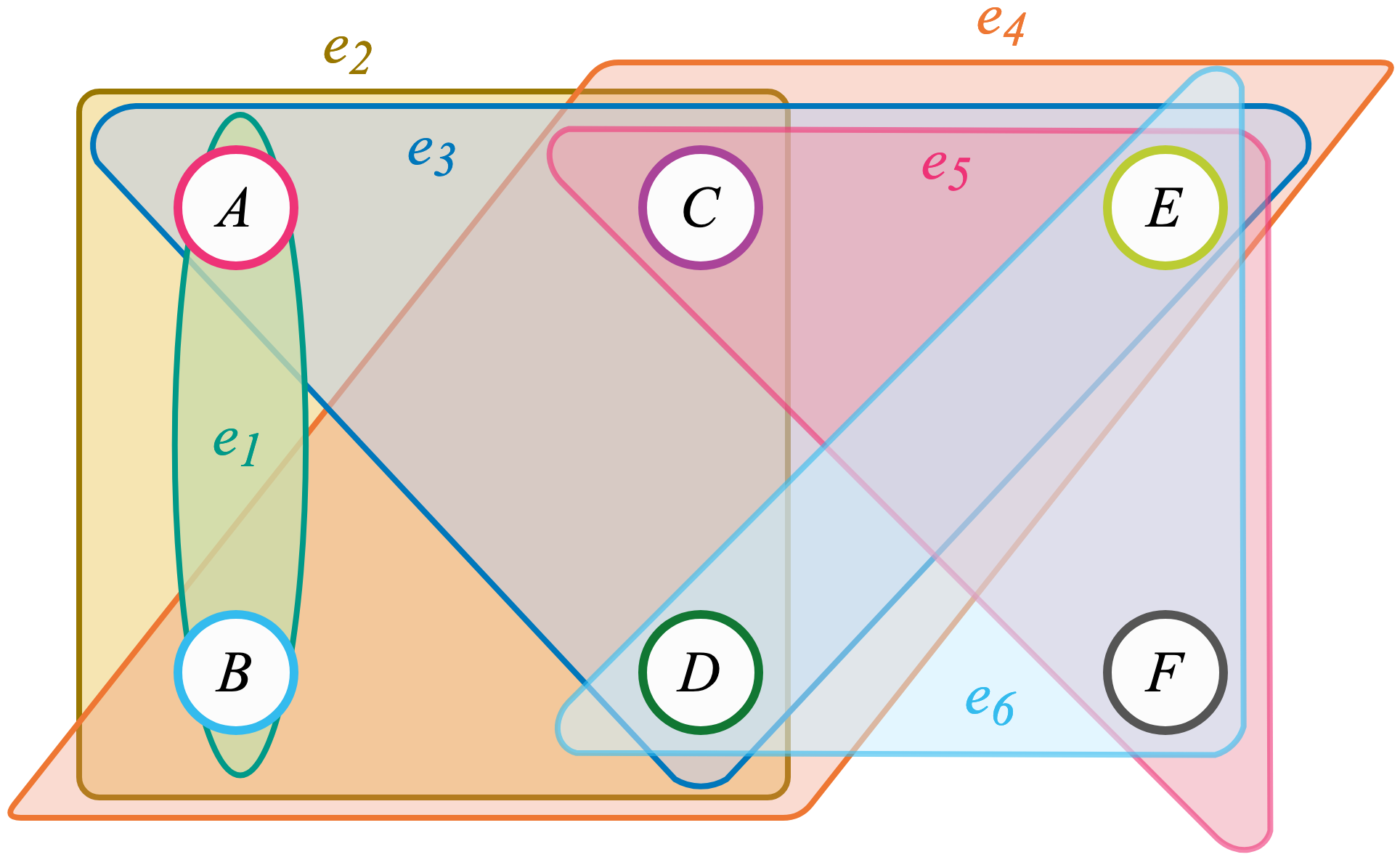}%\label{subfig-1:dummy}
    }
    \subfloat[Hypergraph $R$]{%
       \includegraphics[width=0.45\linewidth]{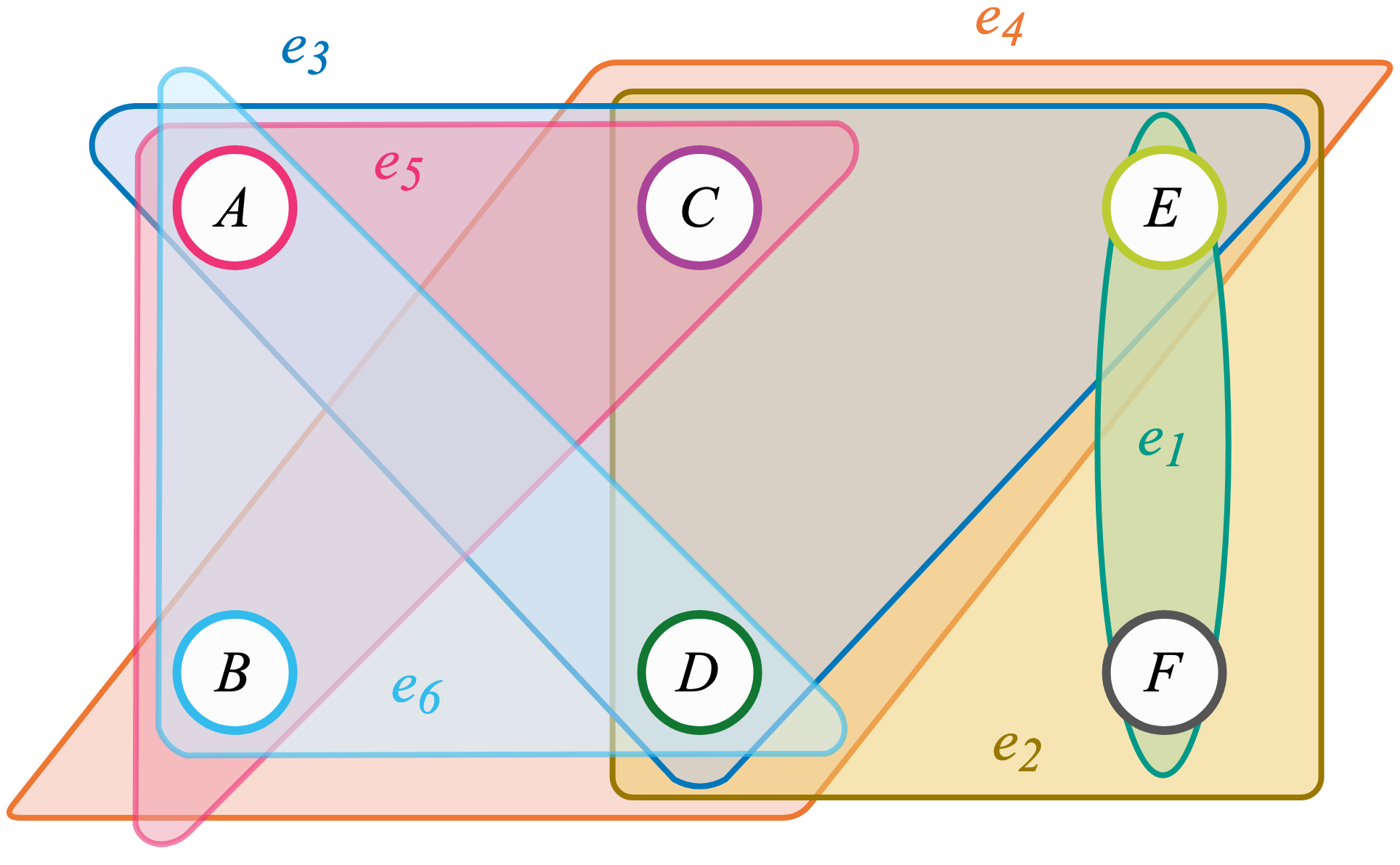}%\label{subfig-2:dummy}
    } \\
        \subfloat[Weighted clique expansion of $S$ and $R$]{%
        \includegraphics[width=0.45\linewidth]{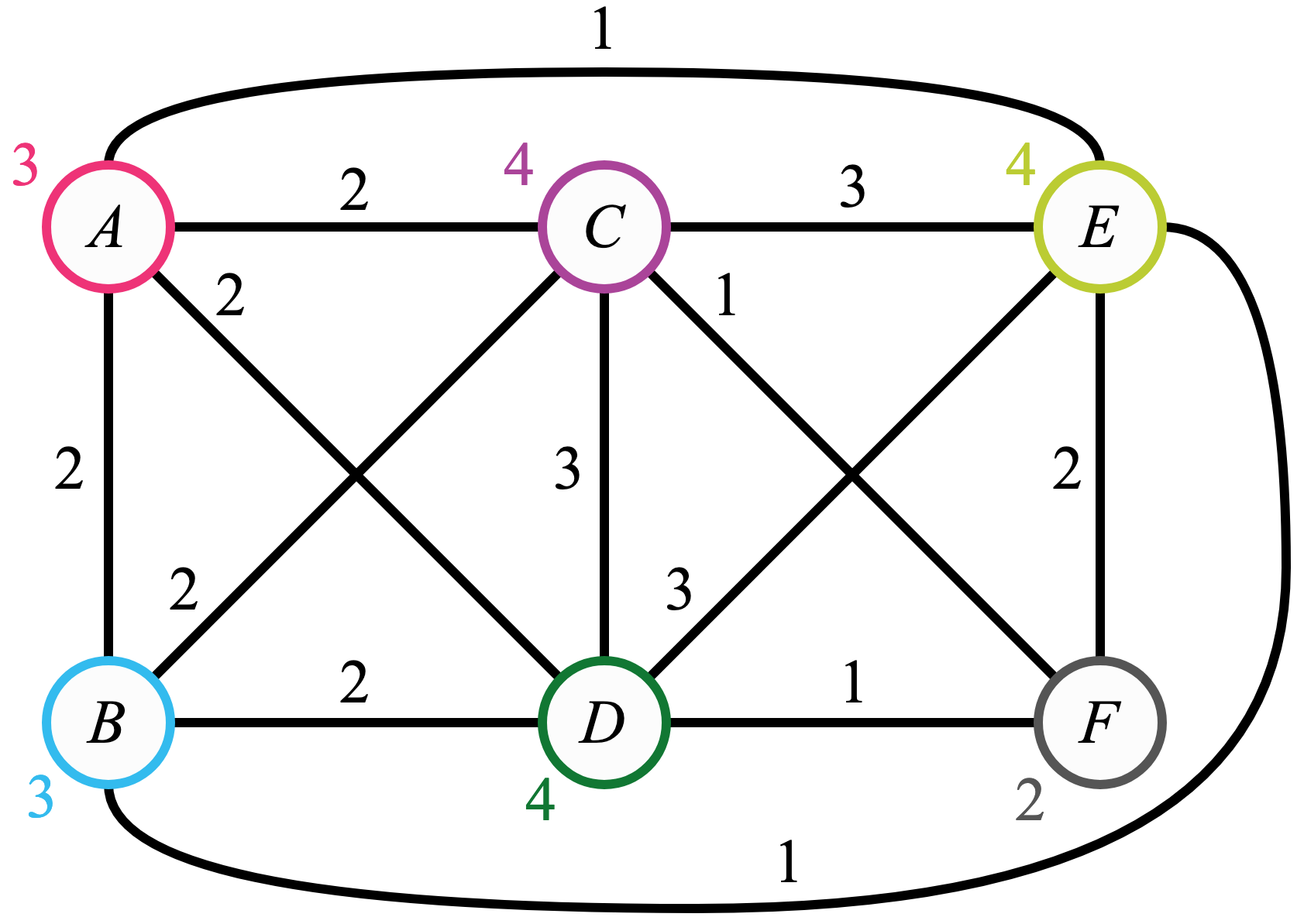}%\label{subfig-1:dummy}
    }
    \subfloat[Weighted line graph of $S$ and $R$]{%
       \includegraphics[width=0.45\linewidth]{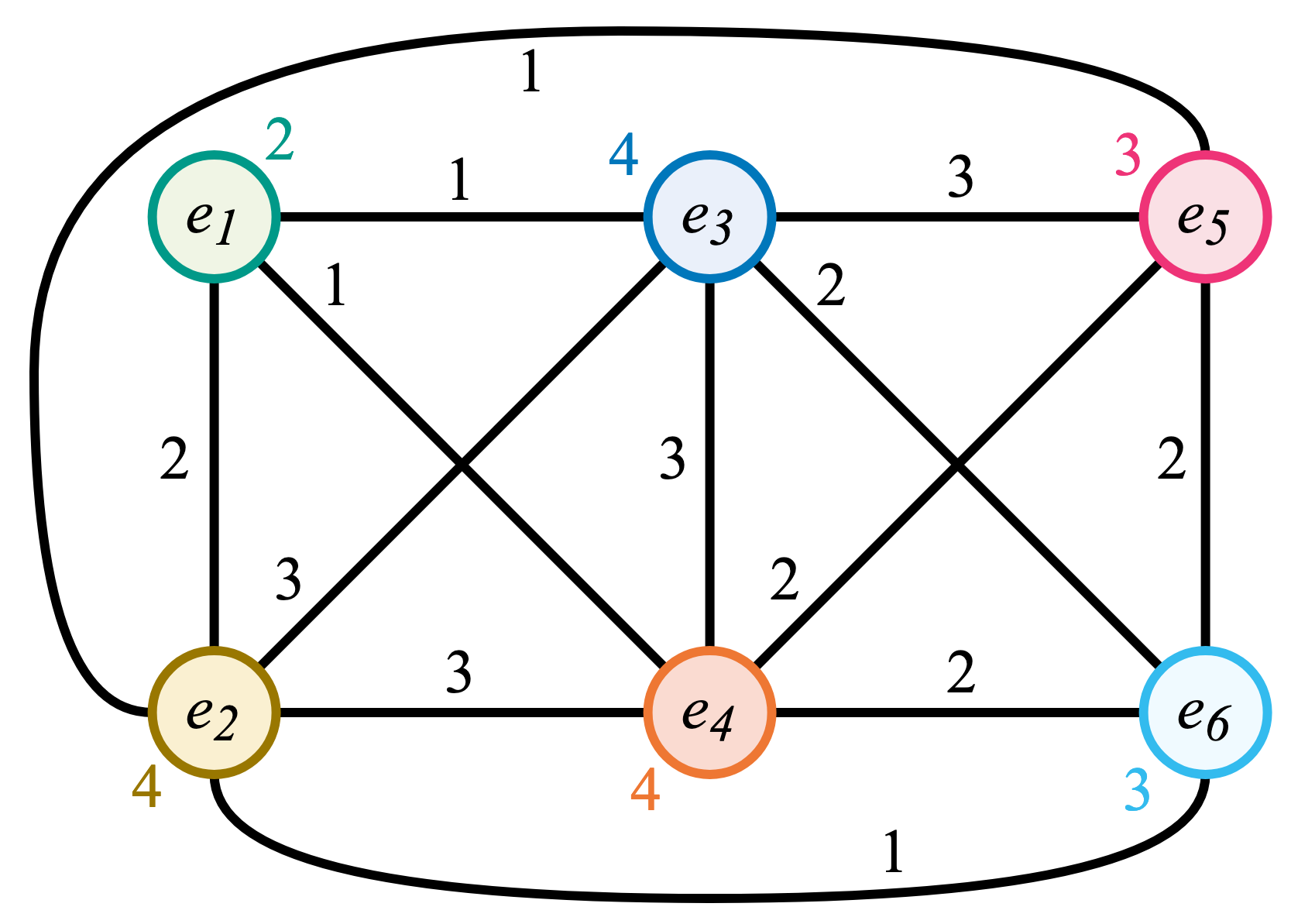}%\label{subfig-2:dummy}
    }
    \caption{Two nonisomorphic hypergraphs with identical weighted clique expansions and identical line graphs.}
    \label{fig:gram}
\end{figure}
\subsubsection{Tensor centrality distinguishes Gram mates}\label{sec:gram}

Having shown the tensor-based ZEC and HEC provide different information than the matrix-based CEC, we now investigate whether ZEC and HEC capture higher-order structure that is {\it inexpressible} by the hypergraph's clique expansion graph.
To address this more nuanced question, we analyze highly structured families of hypergraphs called {\it Gram mates} \cite{kim2022gram, kirkland2018two}. Gram mates are pairs of hypergraphs having incidence matrices $S$ and $R$ satisfying
\begin{align*}
SS^T &= RR^T, \\
S^TS &= R^T R.
\end{align*}
Interpreted combinatorially, $SS^T=RR^T$ means the codegree of any pair of vertices in $S$ is the same as that in $R$, thereby yielding identical weighted clique expansions. 
Similarly, $S^TS=R^TR$ guarantees each pair of hyperedges has the same intersection cardinality in one hypergraph as in the other, meaning their weighted line graphs are identical. Figure \ref{fig:gram} presents a small example derived from \cite{kim2022gram,mizutani2023information} of non-isomorphic Gram mate hypergraphs alongside their weighted clique expansion and line graphs.  We emphasize many existing hypergraph measures and matrices cannot distinguish between these two hypergraphs. For example:

\begin{itemize}
    \item The singular values of the incidence matrices $S$ and $R$.
    %\item Markov chain analysis based on Zhou's \cite{zhou2006learning} hypergraph random walks.
    \item Bolla \cite{bolla1993spectra} and Rodriguez's \cite{rodri2002laplacian} hypergraph Laplacian, Cardoso's signless Laplacian \cite{cardoso2022signless}, the hypergraph adjacency matrix \cite{ cardoso2022adjacency,sole1996spectra}, and $s$-line graphs \cite{aksoy2020hypernetwork}. 
    \item Gibson's dynamical system for categorical data and hypergraph clustering \cite{gibson2000clustering}.
    \item The hypergraph core/periphery, structural equivalence, and centrality methods derived from the ``dual-projection" approach advocated for in \cite{everett2013dual}.
    \item Bipartite projection based analyses, such as bipartite modularity \cite{arthur2020modularity}.
\end{itemize}

In contrast to the above, ZEC and HEC do distinguish between the two hypergraphs in Figure \ref{fig:gram}. In particular, letting $\Vc{x}_S$ and $\Vc{x}_R$ denote the centrality vectors for either ZEC or HEC applied to hypergraphs $S$ and $R$, we have that 
\begin{align*}
     \Vc{x}_S(u)&=\Vc{x}_R(u) \mbox{ for } u=C,D. \\
    \Vc{x}_S(u)&>\Vc{x}_R(u) \mbox{ for } u=E,F. \\
    \Vc{x}_S(u)&<\Vc{x}_R(u) \mbox{ for } u=A,B.
\end{align*}

\subsection{Clustering}\label{subsec:clustering}

Our \ttsv\ algorithms also enable computation of hypergraph tensor embeddings, which may then serve as features for many clustering algorithms, such as $k$-means. Following this approach, we aim to embed the hypergraph adjacency tensor $\A$ in $\R^{n\times q}$, where $q$ is the target embedding dimension, so that each node is represented by a $q$-dimensional vector. We perform the embedding by finding a symmetric CP-decomposition~\cite{hitchcock1927expression,kolda2015numerical} of $\A$, meaning we seek an $n\times q$ matrix $\Mx{E}$ and a vector $\Vc{\lambda}\in \R^q$ such that the tensor norm given by 
\begin{equation}f(\Vc{\lambda},\Mx{E})=||\A-\T{X}||\mbox{ with }\T{X}=\sum_{j=1}^{q}\Vc{\lambda}_j\Mx{E}_j^{\otimes r}\label{eqn:objective}\end{equation}
is minimized, where $\Mx{E}=[\Mx{E}_1 \Mx{E}_2 \cdots \Mx{E}_q]$ and $\Mx{E}_j^{\otimes r}$ is the $r$-way tensor outer product of $\Mx{E}_j$ with itself. To optimize Eq.~\ref{eqn:objective}, we employ a standard first-order optimization scheme and utilize the  closed-form expressions \cite{kolda2015numerical} for the gradients
\begin{align*}
\frac{\partial f}{\partial \Vc{\lambda}_j} &= -2\left[\A\Mx{E}_j^r-\sum_{k=1}^q\Vc{\lambda}_k\langle \Mx{E}_j,\Mx{E}_k\rangle ^r\right],\\
\frac{\partial f}{\partial \Mx{E}_j} &= -2d\Vc{\lambda}_j\left[\A\Mx{E}_j^{r-1}-\sum_{k=1}^q\Vc{\lambda}_k\langle \Mx{E}_j,\Mx{E}_k\rangle ^{r-1}\Mx{E}_k\right],
\end{align*}
where $\langle \Mx{E}_j,\Mx{E}_k\rangle=\Mx{E}_j^T\Mx{E}_k$, $\A\Mx{E}_j^r=\sum_{i_1=1}^n\cdots\sum_{i_r=1}^n\A_{i_1,\dots,i_r}\prod_{k=1}^r\Mx{E}_{j_{i_k}}$ is the \ttsv\ operation which results in a scalar.
%, and $\A\Mx{E}_j^{r-1}\in\R^n$ is the \ttsv[1] operation given by 
% \[
% \left[\A\Mx{E}_j^{r-1}\right]_{i_1}=\sum_{i_2=1}^n\cdots\sum_{i_r=1}^n\A_{i_1,\dots,i_r}\prod_{k=2}^r\Mx{E}_{j_{i_k}}.
% \]
 Note the \ttsv\ value is obtained from the \ttsv[1] vector simply by taking an inner product with $\Mx{E}_j.$ In the first order scheme, computing $f$ and its derivatives explicitly requires $O(qn^r)$ time, but we use \hyperref[alg:TTSV1G]{\ttsvG[1]}, together with the gradient computation approach outlined in~\cite{sherman2020estimating}, to cut this time down to $O(\ttsv[1]+ nq^2)$, where $O(\ttsv[1])$ is the worst-case runtime of \hyperref[alg:TTSV1G]{\ttsvG[1]}. After obtaining this CP decomposition for $\A$, the resulting embedding $E\in\R^{n\times q}$ may be used as features for a standard $k$-means clustering algorithm~\cite{hartigan1979algorithm} or more generally within any metric-space clustering framework. 
 %where $k$ is the desired number of clusters. 
%In practice, the value of $k$ is chosen to correspond to the number of ground truth clusters whenever such information is available. The pipeline is presented in Figure~\ref{fig:clustering}. 
% Beyond just $k$-means, the aforementioned embedding method is compatible with any metric-space clustering framework. 
\begin{figure}
    \centering
    \begin{tabular}{cc}
        \subfloat[\scriptsize t-SNE of embedding for the normalized Laplacian of the clique expansion (left) vs.~normalized Laplacian tensor of the hypergraph (right) of \texttt{cooking} \label{fig:tsne-cooking}]{\includegraphics[width=0.23\linewidth]{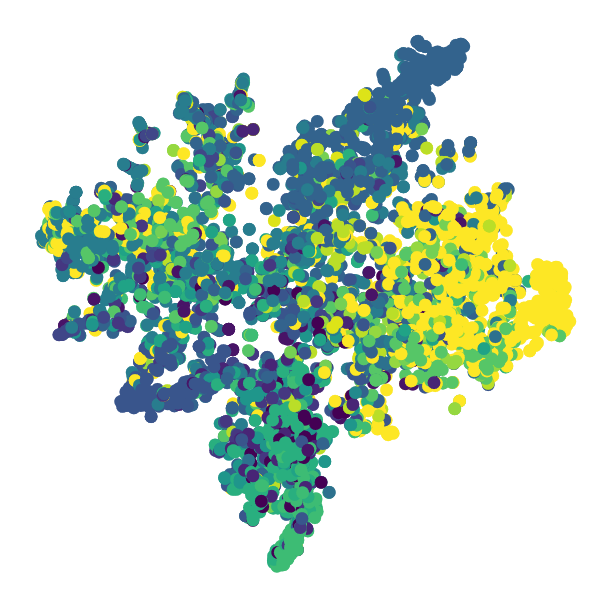}\includegraphics[width=0.23\linewidth]{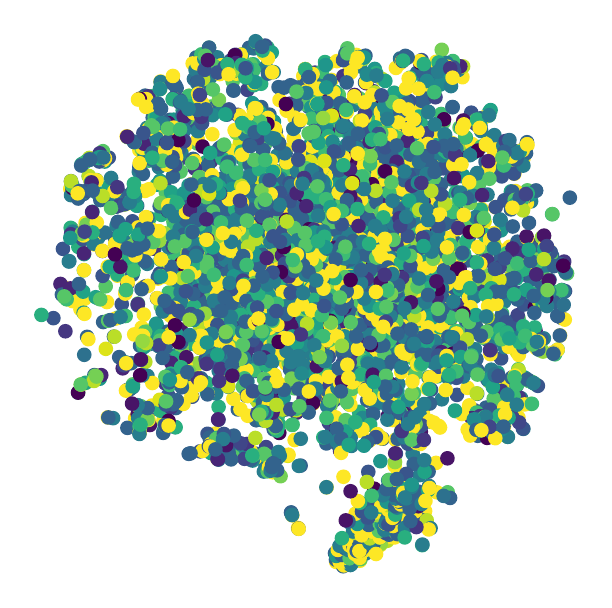}} &  \subfloat[\scriptsize t-SNE of embedding for the normalized Laplacian of the clique expansion (left) vs.~normalized Laplacian tensor of the hypergraph (right) for \texttt{DAWN} \label{fig:tsne-DAWN}]{\includegraphics[width=0.23\linewidth]{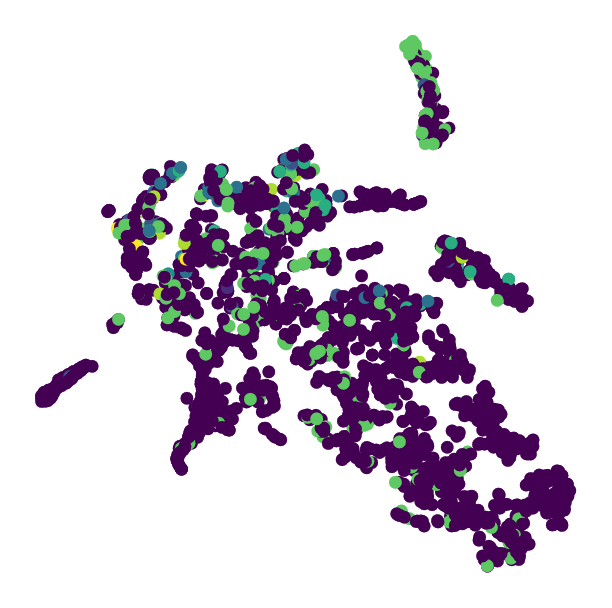}\includegraphics[width=0.23\linewidth]{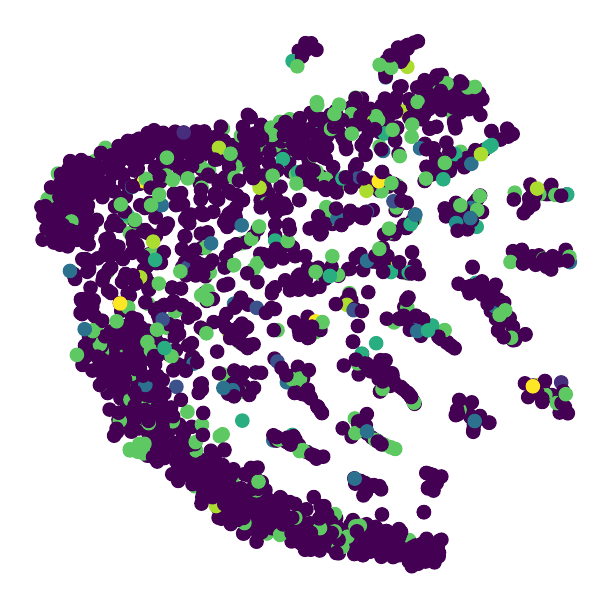}}
    \end{tabular}
    \caption{Comparison of matrix~\cite{zhou2006learning} to tensor embedding of \texttt{cooking} (left) and \texttt{DAWN} (right).}
    \label{fig:global_measures}
\end{figure}
% \begin{figure}
%     \centering
%     \includegraphics[width = 0.49\textwidth]{figs/DAWN.pdf}
%     \includegraphics[width = 0.49\textwidth]{figs/cooking.pdf}
%     \caption{t-SNE visualization of the tensor embedding for the contact-primary and cooking dataset. \sga{What if you did t-SNE plot of clustering based on the clique expansion? Would the clusters not be as well separated?}} %We observe that the clusters are well-separated.}
%     \label{fig:tsne}
% \end{figure}
Instead of clustering $\A$ directly, we cluster the corresponding normalized Laplacian tensor $\LL$ from \cite{BANERJEE201714}, given by
\[
\LL_{p_1\dots p_{r}}=\begin{cases}-\left[\prod_{j=1}^{r}d(v_{p_j})^{-1/r}\right]\frac{\lvert e\rvert}{\lvert \beta(e)\rvert}&\mbox{ if } p_1\dots p_r \in \beta(e) \\
1 &\mbox{ if } p_1 = p_2 = \cdots = p_r \\
0 &\mbox{ otherwise}
\end{cases},
\]
It is worth noting that $\LL$ does not equally weight all blowups of an edge, and so \hyperref[alg:TTSV1gen]{\ttsvG} can not be directly applied.  However, if $\Vc{d}$ is the vector of all degrees we have that $\LL \Vc{x}^r = \T{I} \Vc{x}^r - \A (\Vc{d}^{[-\nicefrac{1}{r}]} \odot \Vc{x})^r$ and, more generally, $\LL \Vc{x}^{r-k} = \T{I} \Vc{x}^{r-1} - \paren{\Vc{d}^{[-\nicefrac{1}{r}]}}^{\otimes k} \odot \A (\Vc{d}^{[-\nicefrac{1}{r}]} \odot \Vc{x})^{r-k},$ where $\odot$ is the appropriate element-wise (Hadamard) tensor product.

We now apply the aforementioned clustering approach to \texttt{cooking} and \texttt{DAWN}. To better reveal clusters, we filter out high-degree nodes that appear in more that 20\% of the hyperedges. We color nodes based on hyperedge type metadata, assigning each node to the majority color of hyperedge it appears in. To  better reveal node colors and speed up the computations, we also filter out hyperedges above size $r=8$. Figures~\ref{fig:tsne-cooking} and~\ref{fig:tsne-DAWN} present t-SNE~\cite{van2008visualizing} plots visualizing embeddings of these datasets.  
For each dataset, the left visualization presents the matrix-based embedding using the normalized Laplacian \cite{chung1997spectral} of the hypergraph's clique expansion graph\footnote{We also used this embedding to initialize the iterative scheme to obtain the tensor embedding.}, while the right shows the normalized Laplacian tensor embedding of the hypergraph.
%t-SNE of the normalized Laplacian embedding for the clique expansion matrix (left) and hypergraph tensor (right)
%The left of each figure shows a visualization of an unweighted version of the matrix-based embedding from~\cite{zhou2006learning}, while the right shows the same using our tensor approach. 
%\color{blue} That we observe well-separated clusters, provides qualitative evidence that the tensor embedding of $\A$ is able to distinguish communities. In fact, our tensor-based approach leads to better cluster separation in cooking. \color{black}
% We demonstrate that our approach is able to detect community structure in the t-SNE plots below. Note that the clusters are well-separated, implying a good embedding to distinguish communities.
%\sga{Perhaps replace with a simpler ``they are different" statement like we discussed: For both datasets, we observe starkly different geometry in their t-SNE plots when comparing the matrix versus tensor embeddings, providing qualitative evidence these approaches provide different information for clustering hypergraph data in practice.}  %In the tensor embedding for cooking and DAWN exhibit striations, whereas }
For both datasets we observe starkly different geometry between the t-SNE representations of the matrix and tensor embeddings, providing qualitative evidence that these two approaches are capturing different features of the hypergraph cluster in practice\footnote{No effort was made to tune the performance of either clustering algorithm or evaluate which clustering is ``better." Indeed, understanding which hypergraph structural features are highlighted by the CP decomposition of $\LL$ is a compelling open question for future work.  In particular, the effectiveness of spectral clustering is explainable in part by the tight connection between Laplacian spectra and the combinatorial properties of the graph captured by the Cheeger inequality.  The authors are unaware of any similar known results for the nonuniform adjacency tensor (Laplacian or otherwise) associated with a hypergraph.}.

\section{Conclusion and future work}\label{sec:conclusion}

We developed a suite of algorithms for performing fundamental tensor operations on the nonuniform hypergraph adjacency tensor. Improving upon approaches that are intractable in time and space complexity, we developed efficient, implicit methods tailored to exploit the nuanced symmetry of the adjacency tensor. We then demonstrated how these algorithms give rise to fundamental tensor-based hypergraph analyses, such as centrality and clustering, which hold promise in capturing hypergraph-native structure over existing matrix-based approaches. Our exploration here is not comprehensive, and many avenues remain for future work. First, we note the hypergraph adjacency tensor we utilized is defined for simple, unweighted, nonuniform hypergraphs. Real data may present multiple hyperedges, weights, vertex multiplicities within a hyperedge, or directionality. Extending our methods to accommodate such cases in a principled manner would be advantageous. Secondly, our application of TTSV algorithms to perform hypergraph analyses is cursory and leaves a number of exciting possibilities to future work: how might one develop multilinear tensor PageRank for nonuniform hypergraphs, or use our algorithms in supervised and semi-supervised machine learning problems, such as node classification and link prediction? Furthermore, it seems plausible that our generating function approach to exploiting symmetry in the hypergraph adjacency tensor could be extended to other tensor operations: for instance, the Tucker decomposition involves repeatedly performing a tensor times same matrix operation \cite{de2000best,jin2022scalable}. Lastly, despite the current approach being tailored to hypergraphs, we believe that generating-function-based tensor algorithms similar to the ones we presented may have utility in general symmetric tensor problems beyond the context of hypergraphs.

\section*{Acknowledgments}
We thank Tammy Kolda, Jiajia Li, and Shruti Shivakumar for helpful discussions on tensor decompositions, and Yosuke Mizutani for his assistance in drawing the hypergraphs visualized in Figure \ref{fig:gram}. We would also like to thank the anonymous referee would pointed out the alternative derivation of the generating function expression for \ttsv[k] based on the relationship $\nabla^k \Tn{X}\Vc{b}^r = \frac{(r-k)!}{r!} \Tn{X} \Vc{b}^{r-k}.$ The authors gratefully acknowledge the funding support from the Applied Mathematics Program within the U.S.\ Department of Energy’s Office of Advanced Scientific Computing Research as part of Scalable Hypergraph Analytics via Random Walk Kernels (SHARWK). Pacific Northwest National Laboratory is operated by Battelle for the DOE under Contract DE-AC05-76RL0-1830. PNNL Information Release PNNL-SA-186918.

\bibliographystyle{siamplain}
\bibliography{main}

\begin{thebibliography}{10}

\bibitem{agarwal2006higher}
{\sc S.~Agarwal, K.~Branson, and S.~Belongie}, {\em Higher order learning with graphs}, in Proceedings of the 23rd international conference on Machine learning, 2006, pp.~17--24.

\bibitem{aksoy2020hypernetwork}
{\sc S.~G. Aksoy, C.~Joslyn, C.~O. Marrero, B.~Praggastis, and E.~Purvine}, {\em Hypernetwork science via high-order hypergraph walks}, EPJ Data Science, 9 (2020).

\bibitem{amburg2020clustering}
{\sc I.~Amburg, N.~Veldt, and A.~Benson}, {\em Clustering in graphs and hypergraphs with categorical edge labels}, in Proceedings of The Web Conference 2020, 2020, pp.~706--717.

\bibitem{antelmi2019simplehypergraphs}
{\sc A.~Antelmi, G.~Cordasco, B.~Kami{\'n}ski, P.~Pra{\l}at, V.~Scarano, C.~Spagnuolo, and P.~Szufel}, {\em Simplehypergraphs. jl—novel software framework for modelling and analysis of hypergraphs}, in Algorithms and Models for the Web Graph: 16th International Workshop, WAW 2019, Brisbane, QLD, Australia, July 6--7, 2019, Proceedings 16, Springer, 2019, pp.~115--129.

\bibitem{arthur2020modularity}
{\sc R.~Arthur}, {\em Modularity and projection of bipartite networks}, Physica A: Statistical Mechanics and its Applications, 549 (2020), p.~124341.

\bibitem{BANERJEE201714}
{\sc A.~Banerjee, A.~Char, and B.~Mondal}, {\em Spectra of general hypergraphs}, Linear Algebra and its Applications, 518 (2017), pp.~14--30, \url{https://doi.org/https://doi.org/10.1016/j.laa.2016.12.022}.

\bibitem{Benson-2018-core}
{\sc A.~Benson and J.~Kleinberg}, {\em Link prediction in networks with core-fringe data},  (2019), pp.~94--104.

\bibitem{benson2019three}
{\sc A.~R. Benson}, {\em Three hypergraph eigenvector centralities}, SIAM Journal on Mathematics of Data Science, 1 (2019), pp.~293--312.

\bibitem{benson2019computing}
{\sc A.~R. Benson and D.~F. Gleich}, {\em Computing tensor {Z}-eigenvectors with dynamical systems}, SIAM Journal on Matrix Analysis and Applications, 40 (2019), pp.~1311--1324.

\bibitem{benson2015tensor}
{\sc A.~R. Benson, D.~F. Gleich, and J.~Leskovec}, {\em Tensor spectral clustering for partitioning higher-order network structures}, in Proceedings of the 2015 SIAM International Conference on Data Mining, SIAM, 2015, pp.~118--126.

\bibitem{benson2017spacey}
{\sc A.~R. Benson, D.~F. Gleich, and L.-H. Lim}, {\em The spacey random walk: A stochastic process for higher-order data}, SIAM Review, 59 (2017), pp.~321--345.

\bibitem{berge1984hypergraphs}
{\sc C.~Berge}, {\em Hypergraphs: combinatorics of finite sets}, vol.~45, Elsevier, 1984.

\bibitem{bolla1993spectra}
{\sc M.~Bolla}, {\em Spectra, euclidean representations and clusterings of hypergraphs}, Discrete Mathematics, 117 (1993), pp.~19--39.

\bibitem{cardoso2022adjacency}
{\sc K.~Cardoso, R.~Del-Vecchio, L.~Portugal, and V.~Trevisan}, {\em Adjacency energy of hypergraphs}, Linear Algebra and its Applications, 648 (2022), pp.~181--204.

\bibitem{cardoso2022signless}
{\sc K.~Cardoso and V.~Trevisan}, {\em The signless {L}aplacian matrix of hypergraphs}, Special Matrices, 10 (2022), pp.~327--342.

\bibitem{carletti2020random}
{\sc T.~Carletti, F.~Battiston, G.~Cencetti, and D.~Fanelli}, {\em Random walks on hypergraphs}, Physical review E, 101 (2020), p.~022308.

\bibitem{chitra2019random}
{\sc U.~Chitra and B.~Raphael}, {\em Random walks on hypergraphs with edge-dependent vertex weights}, in International conference on machine learning, PMLR, 2019, pp.~1172--1181.

\bibitem{chodrow2021hypergraph}
{\sc P.~S. Chodrow, N.~Veldt, and A.~R. Benson}, {\em Generative hypergraph clustering: From blockmodels to modularity}, Science Advances, 7 (2021), p.~eabh1303, \url{https://doi.org/10.1126/sciadv.abh1303}.

\bibitem{chung1997spectral}
{\sc F.~R. Chung}, {\em Spectral graph theory}, vol.~92, American Mathematical Soc., 1997.

\bibitem{cooley1965algorithm}
{\sc J.~W. Cooley and J.~W. Tukey}, {\em An algorithm for the machine calculation of complex fourier series}, Mathematics of computation, 19 (1965), pp.~297--301.

\bibitem{de2000best}
{\sc L.~De~Lathauwer, B.~De~Moor, and J.~Vandewalle}, {\em On the best rank-1 and rank-(r 1, r 2,..., rn) approximation of higher-order tensors}, SIAM journal on Matrix Analysis and Applications, 21 (2000), pp.~1324--1342.

\bibitem{everett2013dual}
{\sc M.~G. Everett and S.~P. Borgatti}, {\em The dual-projection approach for two-mode networks}, Social networks, 35 (2013), pp.~204--210.

\bibitem{gibson2000clustering}
{\sc D.~Gibson, J.~Kleinberg, and P.~Raghavan}, {\em Clustering categorical data: An approach based on dynamical systems}, The VLDB Journal, 8 (2000), pp.~222--236.

\bibitem{hartigan1979algorithm}
{\sc J.~A. Hartigan and M.~A. Wong}, {\em Algorithm as 136: A k-means clustering algorithm}, Journal of the royal statistical society. series c (applied statistics), 28 (1979), pp.~100--108.

\bibitem{hayashi2020hypergraph}
{\sc K.~Hayashi, S.~G. Aksoy, C.~H. Park, and H.~Park}, {\em Hypergraph random walks, {L}aplacians, and clustering}, in Proceedings of the 29th ACM International Conference on Information \& Knowledge Management, 2020, pp.~495--504.

\bibitem{hitchcock1927expression}
{\sc F.~L. Hitchcock}, {\em The expression of a tensor or a polyadic as a sum of products}, Journal of Mathematics and Physics, 6 (1927), pp.~164--189.

\bibitem{ismail2009convolution}
{\sc L.~Ismail}, {\em Convolution on top of the {I}{B}{M} cell {B}{E} {P}rocessor.}, in PDPTA, 2009, pp.~507--611.

\bibitem{jin2022scalable}
{\sc R.~Jin}, {\em Scalable symmetric {T}ucker tensor decomposition}, in Fourteenth International Conference on Sampling Theory and Applications, 2023, \url{https://openreview.net/forum?id=17iCMIE0V9}.

\bibitem{ke2019community}
{\sc Z.~T. Ke, F.~Shi, and D.~Xia}, {\em Community detection for hypergraph networks via regularized tensor power iteration}, Jan. 2020, \url{https://arxiv.org/abs/1909.06503v2}.

\bibitem{kim2022gram}
{\sc S.~Kim and S.~Kirkland}, {\em Gram mates, sign changes in singular values, and isomorphism}, Linear Algebra and its Applications, 644 (2022), pp.~108--148.

\bibitem{kirkland2018two}
{\sc S.~Kirkland}, {\em Two-mode networks exhibiting data loss}, Journal of Complex Networks, 6 (2018), pp.~297--316.

\bibitem{kolda2015numerical}
{\sc T.~G. Kolda}, {\em Numerical optimization for symmetric tensor decomposition}, Mathematical Programming, 151 (2015), pp.~225--248.

\bibitem{kolda2009tensor}
{\sc T.~G. Kolda and B.~W. Bader}, {\em Tensor decompositions and applications}, SIAM review, 51 (2009), pp.~455--500.

\bibitem{kolda2011shifted}
{\sc T.~G. Kolda and J.~R. Mayo}, {\em Shifted power method for computing tensor eigenpairs}, SIAM Journal on Matrix Analysis and Applications, 32 (2011), pp.~1095--1124.

\bibitem{landry2023filtering}
{\sc N.~W. Landry, I.~Amburg, M.~Shi, and S.~G. Aksoy}, {\em Filtering higher-order datasets}, May 2023, \url{https://arxiv.org/abs/2305.06910}.

\bibitem{landry2023xgi}
{\sc N.~W. Landry, M.~Lucas, I.~Iacopini, G.~Petri, A.~Schwarze, A.~Patania, and L.~Torres}, {\em {X}{G}{I}: A {P}ython package for higher-order interaction networks}, Journal of Open Source Software, 8 (2023), p.~5162.

\bibitem{sole1996spectra}
{\sc W.-C.~W. Li and P.~Sol{\'e}}, {\em Spectra of regular graphs and hypergraphs and orthogonal polynomials}, European Journal of Combinatorics, 17 (1996), pp.~461--477.

\bibitem{lotito2023hypergraphx}
{\sc Q.~F. Lotito, M.~Contisciani, C.~De~Bacco, L.~Di~Gaetano, L.~Gallo, A.~Montresor, F.~Musciotto, N.~Ruggeri, and F.~Battiston}, {\em Hypergraphx: a library for higher-order network analysis}, Journal of Complex Networks, 11 (2023), p.~cnad019.

\bibitem{mizutani2023information}
{\sc Y.~Mizutani, S.~Aksoy, E.~Alhajjar, J.~Aslam, B.~S. Chrisman, S.~Days-Merrill, H.~Jenne, T.~G. Oellerich, and E.~Semrad}, {\em Information loss in weighted hypergraph line graphs and clique expansions}, in 2023 Joint Mathematics Meetings (JMM 2023), AMS, 2023.

\bibitem{ng2010finding}
{\sc M.~Ng, L.~Qi, and G.~Zhou}, {\em Finding the largest eigenvalue of a nonnegative tensor}, SIAM Journal on Matrix Analysis and Applications, 31 (2010), pp.~1090--1099.

\bibitem{ni2019justifying}
{\sc J.~Ni, J.~Li, and J.~McAuley}, {\em Justifying recommendations using distantly-labeled reviews and fine-grained aspects}, in Proceedings of the 2019 Conference on Empirical Methods in Natural Language Processing and the 9th International Joint Conference on Natural Language Processing (EMNLP-IJCNLP), 2019, pp.~188--197.

\bibitem{ouvrard2017adjacency}
{\sc X.~Ouvrard, J.-M.~L. Goff, and S.~Marchand-Maillet}, {\em Adjacency and tensor representation in general hypergraphs part 1: e-adjacency tensor uniformisation using homogeneous polynomials}, 2018, \url{https://arxiv.org/abs/1712.08189v5}.

\bibitem{pereira2022tensor}
{\sc J.~M. Pereira, J.~Kileel, and T.~G. Kolda}, {\em Tensor moments of {G}aussian mixture models: Theory and applications}, 2022, \url{https://arxiv.org/abs/2202.06930}.

\bibitem{praggastis2023hypernetx}
{\sc B.~Praggastis, S.~Aksoy, D.~Arendt, M.~Bonicillo, C.~Joslyn, E.~Purvine, M.~Shapiro, and J.~Y. Yun}, {\em Hyper{N}et{X}: A python package for modeling complex network data as hypergraphs}, arXiv preprint arXiv:2310.11626,  (2023).

\bibitem{qi2017tensor}
{\sc L.~Qi and Z.~Luo}, {\em Tensor analysis: spectral theory and special tensors}, SIAM, 2017.

\bibitem{rodri2002laplacian}
{\sc J.~A. Rodr{\'{i}}guez}, {\em On the {L}aplacian eigenvalues and metric parameters of hypergraphs}, Linear and Multilinear Algebra, 50 (2002), pp.~1--14.

\bibitem{sharmajoint}
{\sc A.~Sharma and J.~Srivastava}, {\em Joint symmetric tensor decomposition for hypergraph embeddings}.

\bibitem{sherman2020estimating}
{\sc S.~Sherman and T.~G. Kolda}, {\em Estimating higher-order moments using symmetric tensor decomposition}, SIAM Journal on Matrix Analysis and Applications, 41 (2020), pp.~1369--1387.

\bibitem{Sinha-2015-MAG}
{\sc A.~Sinha, Z.~Shen, Y.~Song, H.~Ma, D.~Eide, B.-J.~P. Hsu, and K.~Wang}, {\em An overview of {M}icrosoft {A}cademic {S}ervice ({MAS}) and applications}, in Proceedings of the 24th International Conference on World Wide Web, {ACM} Press, 2015, \url{https://doi.org/10.1145/2740908.2742839}.

\bibitem{van2008visualizing}
{\sc L.~Van~der Maaten and G.~Hinton}, {\em Visualizing data using t-{S}{N}{E}.}, Journal of machine learning research, 9 (2008).

\bibitem{Veldt-2020-local}
{\sc N.~Veldt, A.~R. Benson, and J.~Kleinberg}, {\em Minimizing localized ratio cut objectives in hypergraphs}, in Proceedings of the 26th {ACM} {SIGKDD} International Conference on Knowledge Discovery and Data Mining, {ACM} Press, 2020.

\bibitem{2020SciPy-NMeth}
{\sc P.~Virtanen et~al.}, {\em {{SciPy} 1.0: Fundamental Algorithms for Scientific Computing in Python}}, Nature Methods, 17 (2020), pp.~261--272, \url{https://doi.org/10.1038/s41592-019-0686-2}.

\bibitem{wilf}
{\sc H.~S. Wilf}, {\em generatingfunctionology}, CRC press, 2005.

\bibitem{zhen2021community}
{\sc Y.~Zhen and J.~Wang}, {\em Community detection in general hypergraph via graph embedding}, Journal of the American Statistical Association,  (2022), pp.~1--10.

\bibitem{zhou2006learning}
{\sc D.~Zhou, J.~Huang, and B.~Sch{\"o}lkopf}, {\em Learning with hypergraphs: Clustering, classification, and embedding}, Advances in neural information processing systems, 19 (2006), pp.~1601--1608.

\end{thebibliography}
\makeatletter\@input{xxs.tex}\makeatother
\end{document}

\maketitle

\section{Extension to \ttsv[2] and beyond}\label{sec:TTSV2}
As noted in the main text, our approach readily extends to tensor-times-same-vector in all but $k$.  To illustrate the necessary changes, here we provide the algorithms and analysis for $k =2$, that is \ttsv[2].   To this end, we first provide the natural generalization of $\phi_1$, the scaling value used in \hyperref[alg:TTSV1U]{\ttsvU[1]}, to higher $k$.

\begin{lemma}\label{lem:comb_TTSV2}
Let $e=\{v_1,\dots,v_k\}$ be a hyperedge of a rank $r$ and $\beta(e)$ and $\kappa(e)$ as defined in Definition $\ref{def:blowup}$. Let $x \in \kappa(e)$ and let $u_1,u_2,\ldots, u_k \in x$ such that $\set{u_1,\ldots,u_k} \subseteq x$ as a multiset.  The number of blowups $i_1,\ldots, i_r \in \beta(e)$ that have support equal to $x$ with $i_j = u_j$ for $1 \leq j \leq k$, is given by the multinomial coefficient
\[ \phi_k(x,u) \coloneqq \binom{r-k}{m_x(v_1) - m_u(v_1), \ldots, m_x(v_k) - m_u(v_k)}, \]
where $m_x(w)$ and $m_u(w)$ denotes the multiplicity of $w$ in $x$ and $u$, respectively.
\end{lemma}

This immediately yields the generalization shown in \hyperref[alg:TTSV2U]{\ttsvU[2]} of the implicit \ttsv\ approach to the case $k =2$. Applying the same observations which yielded \hyperref[alg:TTSV1G]{\ttsvG}  to this algorithm yields the generating function approach \hyperref[alg:TTSV2G]{\ttsvG[2]}.
\begin{center}
\phantomsection\label{alg:TTSV2U}
\begin{algorithm}[H] \small
\LinesNotNumbered 
\KwData{rank $r$ hypergraph $(V,E,w)$, vector $\Vc{b}$}
\KwResult{$\A \Vc{b}^{r-2}=\Mx{Y}$}
\For{$(u,v)\in V^2$ with $u\leq v$}{
$c \gets 0$\;

\For{$e\in E(u,v)$}{
\For{$x \in \kappa(e)$}{
\If{$u = v$ \emph{and} $m_x(u)=1$}{
\Continue
}
$c \pluseq \displaystyle w_e \frac{\phi_2(x,u,v)}{\Vc{b}_u\Vc{b}_v}\cdot \prod\limits_{u \in x}\Vc{b}_u$\;}
}
$\Mx{Y}_{uv} \gets c$\;
}
\Return{$\Mx{Y}+\Mx{Y}^T-\mathrm{diag}(\Mx{Y})$}
\caption*{\ttsvU[2]: TTSV2 via unordered blowups}
\end{algorithm}
\end{center}

\begin{center}
    \begin{minipage}[t]{0.9\linewidth}
    \phantomsection\label{alg:TTSV2G}
\begin{algorithm}[H] \small
\LinesNotNumbered 
\KwData{rank $r$ hypergraph $(V,E)$, vector $\Vc{b}$}
\KwResult{$\A \Vc{b}^{r-2}=\Mx{Y}$}
\For{$(u,v)\in V^2$ with $u\leq v$}{
$c \gets 0$\;

\For{$e\in E(u,v)$}{
\If{$u = v$}{
$c\pluseq \displaystyle w_e (r-2)! \paren{\exp({\Vc{b}_ut}) \prod_{ \substack{ x \in e \\ x \neq u}}(\exp({\Vc{b}_ut}) -1)}[t^{r-2}]$\;
}
\Else{
$c\pluseq \displaystyle w_e (r-2)! \paren{\exp({(\Vc{b}_u + \Vc{b}_v)t}) \prod_{ \substack{ x \in e \\ u \neq u,v}}(\exp({\Vc{b}_ut}) -1)}[t^{r-2}]$\;
}
$\Mx{Y}_{uv} \gets  c$\;
}
}
\Return{$\Mx{Y}+\Mx{Y}^T-\mathrm{diag}(\Mx{Y})$}
\caption*{\ttsvG[2]: \ttsv[2] via generating functions}
\end{algorithm}
\end{minipage}
%\caption{Implicit TTSV1 (top) and TTSV2 (bottom) for the hypergraph adjacency tensor $\A$ via the generating function approach.}
\end{center}

\begin{proposition}
Let $H$ be a rank $r$ hypergraph with $m$ edges, let $\delta = \frac{\Vol(H)}{m}$ be the average edge size, let $\epsilon = \min\set{
\frac{\delta}{r}, 1- \frac{\delta}{r}}$, amd let $k^* = \log_2(r) + \log_2 \log_2(r)$. \hyperref[alg:TTSV2U]{\ttsvU[2]} runs in time at least 
\[ \bigOmega{r^2 \Vol(H)} \quad \textrm{and at most} \quad \bigOh{\epsilon m r^{\nicefrac{5}{2}}2^r}.\]
Further, if there is some constant $\tau$ such that $2 + \tau < \delta < (1-\tau)r$, then  \hyperref[alg:TTSV2G]{\ttsvG[2]} runs in time at least 
\[\begin{cases} \bigOmega{\Vol(H)\delta 2^{\delta}\log_2(r)} & \delta < \frac{1}{2}k^* \\ \bigOmega{\Vol(H)\delta 2^{(1-\lilOh{1})\delta}\log_2(r)} & \frac{1}{2} k^* \leq \delta \leq k^*  \\
\bigOmega{\Vol(H)\delta^2 r\log_2(r)} & k^* < \delta
\end{cases} \]
and at most
\[\bigOh{r^{3}\log_2(r)\Vol(H)} .\]
\end{proposition}

\begin{proof}
We first consider \hyperref[alg:TTSV2U]{\ttsvU[2]} and observe that this algorithm iterates over all subsets of size at most 2 in an edge, and thus we have that the per-edge run time is given by 
\[ r \paren{ \binom{\size{e}}{2} + \size{e}} \binom{r-1}{\size{e} - 1} = \frac{\size{e}^2 + \size{e}}{2} \size{e} \binom{r}{\size{e}}.\]
With minor adjustments, the upper bound argument for \hyperref[alg:TTSV1U]{\ttsvU} presented in \cref{P:TTSV} can be used to provide the stated upper bound on the runtime.

To provide the lower bound for \hyperref[alg:TTSV2U]{\ttsvU[2]}, we note that the run time for an edge of size 2 is $3r(r-1)$, of size 3 is $3r(r-1)(r-2)$, and of size $r$ is $\frac{r^3+r^2}{2}$, while all other edge sizes run in time $\bigOmega{r^4}$.  
Thus, the minimal total run time is achieved by maximizing the number of edges of size two while achieving the desired volume by including edges of size $r$.  This results in a total run time 
\[ \frac{rm - \Vol(H)}{r-2} 3r(r-1) + \frac{\Vol(H) - 2m}{r-2} \frac{r^3+r^2 }{2} = 2r^2m + \frac{r(r-3)}{2}\Vol(H) \]
Thus, the overall lower bound on the runtime of \hyperref[alg:TTSV2U]{\ttsvU[2]} is $\bigOmega{r^2 \Vol(H)}$,
as desired.

The analysis for \hyperref[alg:TTSV2G]{\ttsvG[2]} follows similar lines as \cref{P:TTSVgen}.
\end{proof}

For \ttsv[2], we observe empirical runtime performance similar to that for \ttsv[1], which is presented in~\cref{fig:ttsv2-times}. In particular, the two baselines (explicit and ordered) largely time out or force an out-of-memory error for low $r$ values, while the generating function approach drastically outperforms both the baselines and unordered approach. For example, on \texttt{mathoverflow} the generating function approach outperforms the unordered approach by about an order of magnitude at $r = 8$, and that increases to over two orders of magnitude at $r=16$. The notable exception is \texttt{amazon-reviews}, on which the unordered approach slightly outperforms the generating function approach for $r = 4$.
\begin{figure}
    \centering
    \includegraphics[width=0.99\linewidth]{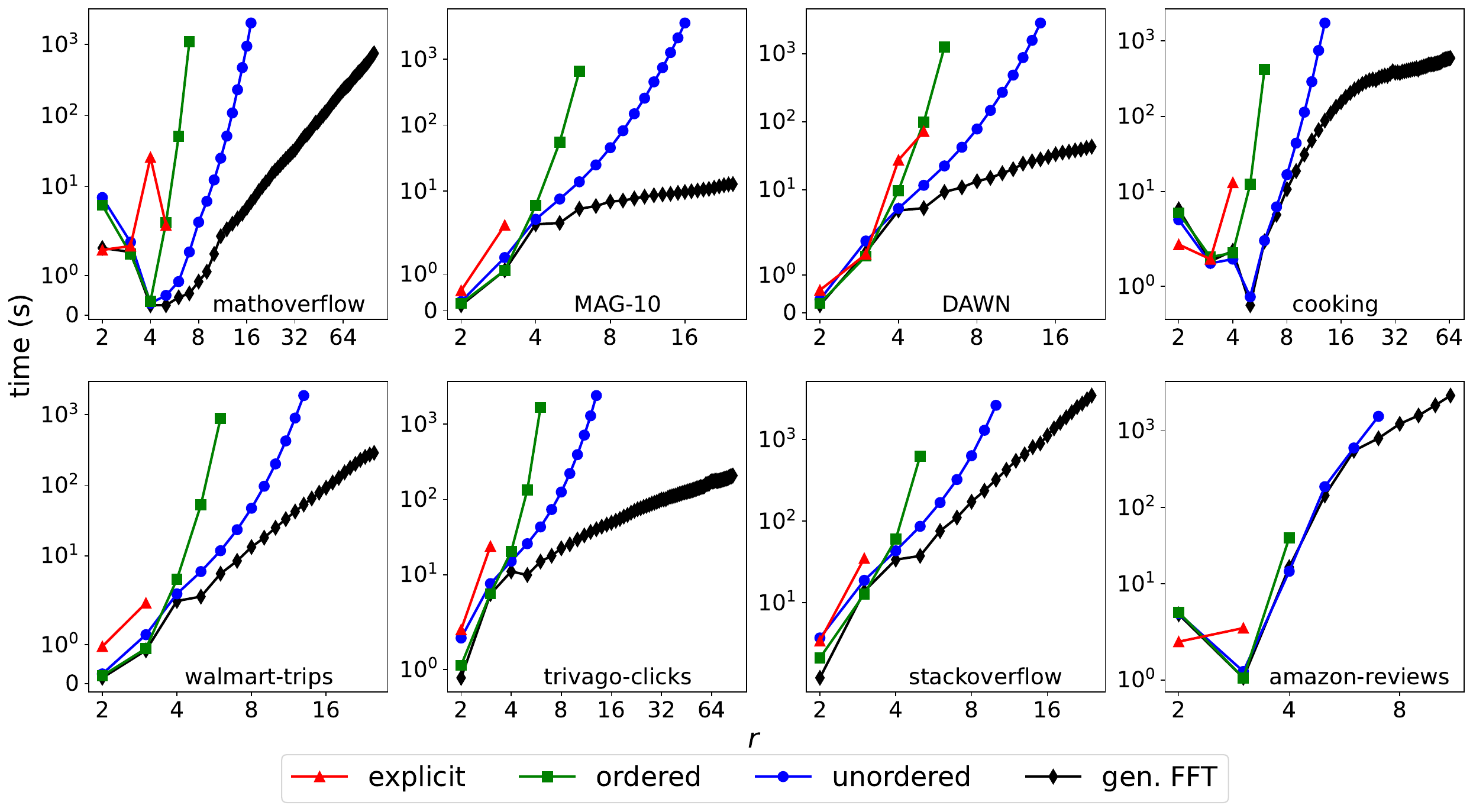}
    \caption{
    Runtimes of TTSV2 algorithms.
    %  Runtimes of Algorithms~\ref{alg:TTSV1gene} (left) and \ref{alg:TTSV2gene} (right).
    }
    \label{fig:ttsv2-times}
\end{figure}

\section{Numerical considerations}\label{sec:numerical}
\begin{figure}
    \centering
    \includegraphics[width=0.60\linewidth]{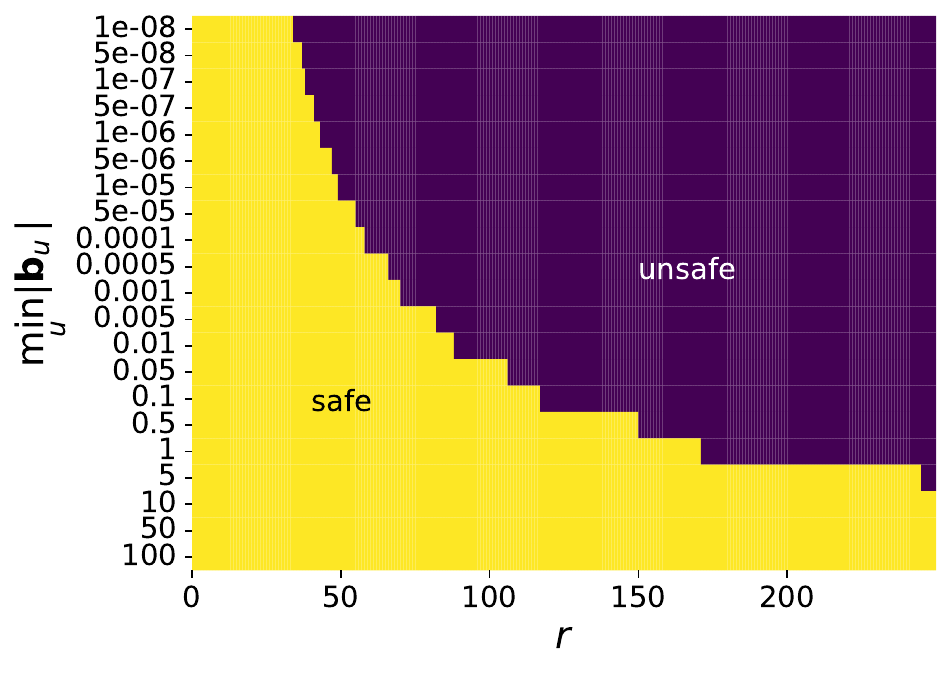}
    \caption{Safe and unsafe parameter regimes for TTSV1 and TTSV2 algorithms.}
    \label{fig:safe-float}
\end{figure}
Here, we address the issue of potential roundoff error in the generating function approach to \ttsv calculations. Both algorithms compute coefficients of the form $\Vc{b}_i^k/k!$, where $k\leq r$. For large $r$ and $k$ close to $r$, and small vector component $\Vc{b}_i$, we must be careful to avoid roundoff errors. In particular, when running our algorithms the user should ascertain that the values of these coefficients will not reach below the smallest (in magnitude) non-subnormal value representable in their system. While this value is system-dependent, in~\cref{fig:safe-float} we present a plot of the ranges of $\min_u\Vc{b}_u$ and $r$ for which the coefficient $\min_u\Vc{b}_u^r/r!$ safely exceeds (yellow region) the value for our system (which happens to be around $2\times 10 ^{-308}$). For example, $\min_u\Vc{b}_u = 0.05, r = 100$ safely falls in the yellow region -- and this corresponds to the extremal setting used in our timing experiments presented in \cref{tab:allexp1}.

\section{Case study: food and drugs}\label{sec:case-study}
We present further case study on how choice of $r$ greatly affects the ranking of top nodes. In Table~\ref{tab:cooking_top_ranked}, we observe the rankings within \texttt{cooking} change drastically both across the three centrality measures, and across the two filterings. This suggests the centrality measures present complementary results, and incorporating larger hyperedges leads to fundamentally different rankings. For instance, we note the top-ranked ingredients for all three centralities transition from what looks like a 3-ingredient recipes for foundational cuisine elements such as corn tortillas (water, salt, masa harina), dashi (water, konbu, dried bonito flakes), and rice (water, salt, long-grain rice) to something resembling a savory pastry recipe that requires more ingredients (all-purpose flour, salt, olive oil, water, butter, etc.). While the analogous changes in ranking across filtering parameter and centrality are not as apparent in \texttt{DAWN}, each centrality still induces a different ranking, and new substances such as amphetamine make their way to the list for $r=22.$
\clearpage

\begin{table}[ht!]
    \centering
    \scalebox{0.65}{
    \begin{tabular}{l| c c c| c c c }
    \toprule
    \multicolumn{4}{c|}{ $r = 3$} & \multicolumn{3}{c}{ $r = r_{\max} = 65$}\\
    \midrule
    \textit{rank} & CEC & ZEC & HEC  & CEC & ZEC & HEC\\
    \midrule
    1& water                & water               & water                   & salt                 & salt              & salt                   \\
    2& salt                 & salt                & salt                    & onions                 & olive oil                            & olive oil       \\
    3& sugar                & masa harina         & sugar                   & olive oil                 & onions                            & onions                   \\
    4& butter               & peanuts             & peanuts                 & garlic     & water         & garlic                   \\
    5& peanuts              & polenta             & masa harina             & water                & garlic                    & water                 \\
    6& lemon                & white rice          & butter                  & garlic cloves           & sugar                           & sugar                   \\
    7& masa harina          & flour               & polenta                 & pepper               & garlic cloves                             & all-purpose flour                   \\
    8& kosher salt          & long-grain rice     & konbu                   & sugar                 & pepper                  & garlic cloves        \\
    9& konbu                & butter              & dried bonito flakes     & ground black pepper                   & all-purpose flour                            & pepper               \\
    10& polenta             & yellow corn meal    & kosher salt             & all-purpose flour                    & butter                         & butter                  \\
    \bottomrule
    \end{tabular}
    }
\newline
\vspace*{0.5 cm}
\newline
\hspace*{-0.8cm}
%     \caption{Top ranked ingredients in \texttt{cooking} for $r =3,r_{\max}$ across CEC, ZEC, and HEC.}
%     \label{tab:cooking_top_ranked}
% % \end{table}
% % \begin{table}[h]
    %\centering
    \scalebox{0.5}{
    \begin{tabular}{l| c c c| c c c }
    \toprule
    \multicolumn{4}{c|}{ $r = 3$} & \multicolumn{3}{c}{ $r = r_{\max} = 22$}\\
    \midrule
    \textit{rank} & CEC & ZEC & HEC  & CEC & ZEC & HEC\\
    \midrule
    1& alcohol (ethanol)             & alcohol (ethanol)                & alcohol (ethanol)                  & alcohol (ethanol)             & alcohol (ethanol)                & alcohol (ethanol)       \\
    2& cocaine                       & cocaine                          & cocaine                            & cocaine                       & cocaine                          & cocaine                 \\
    3& marijuana                     & marijuana                        & marijuana                          & marijuana                     & marijuana                        & marijuana               \\
    4& heroin                        & heroin                           & heroin                             & heroin                        & heroin                           & heroin                  \\
    5& alprazolam                    & methamphetamine                  & methamphetamine                    & benzodiazepines-nos           & narcotic analgesics-nos          & benzodiazepines-nos     \\
    6& narcotic analgesics-nos       & drug unknown                     & drug unknown                       & alprazolam                    & benzodiazepines-nos              & narcotic analgesics-nos\\
    7& benzodiazepines-nos           & narcotic analgesics-nos          & narcotic analgesics-nos            & narcotic analgesics-nos & methamphetamine & methamphetamine\\
    8& methamphetamine               & benzodiazepines-nos              & benzodiazepines-nos                & methamphetamine & alprazolam & alprazolam\\
    9& drug unknown                  & alprazolam                       & alprazolam                         & methadone & drug unknown & drug unknown\\
    10& methadone                    & methadone                        & methadone                          & drug unknown & amphetamine & methadone\\
    \bottomrule
    \end{tabular}
    }
    \caption{Top ranked ingredients in \texttt{cooking} (top) and drugs in \texttt{DAWN} (bottom) for $r =3,r_{\max}$ across CEC, ZEC, and HEC.}
    \label{tab:cooking_top_ranked}
\end{table}

% \section[Proof of Thm]{Proof of \cref{thm:bigthm}}
% \label{sec:proof}

% \lipsum[106-112]

% \bibliographystyle{siamplain}
% \bibliography{references}

\makeatletter\@input{xx.tex}\makeatother